\documentclass{amsart}

\usepackage{graphicx}
\usepackage{wrapfig}
\usepackage{extarrows} 
\usepackage{relsize}
\usepackage{tikz}
\usetikzlibrary{cd}
\usepackage{amsthm} 
\usepackage{amsfonts}
\usepackage{amssymb,latexsym, mathrsfs}
\usepackage{amsmath, mathtools}  
\usepackage{verbatim,enumerate}
\usepackage{xcolor}
\usepackage{hyperref}
\usepackage[utf8]{inputenc}
\usepackage[noabbrev,capitalise]{cleveref}

\numberwithin{equation}{section}

\theoremstyle{plain}
\newtheorem{theorem}{Theorem}[section]

\newtheorem*{theorem*}{Theorem}

\newtheorem{lemma}[theorem]{Lemma}
\newtheorem*{lemma*}{Lemma}
\newtheorem{proposition}[theorem]{Proposition}
\newtheorem*{proposition*}{Proposition}
\newtheorem{corollary}[theorem]{Corollary}
\newtheorem*{corollary*}{Corollary}

\newtheorem*{theoreme*}{Théorème}
\newtheorem*{lemme*}{Lemme}
\newtheorem*{corollaire*}{Corollaire}
\newtheorem*{definitionf*}{Définition}

\theoremstyle{definition} 
\newtheorem{definition}[theorem]{Definition}
\newtheorem*{definition*}{Definition}
\newtheorem{example}[theorem]{Example}
\newtheorem{remark}[theorem]{Remark}

\DeclareMathOperator{\En}{End}
\DeclareMathOperator{\EEn}{\mathcal{E}\hspace*{-0,5mm}\text{nd}}
\DeclareMathOperator{\Aut}{Aut}

\DeclareMathOperator{\Mod}{mod}
\DeclareMathOperator{\proj}{proj}
\DeclareMathOperator{\per}{per}
\DeclareMathOperator{\pero}{per^{[-1,0]}}
\DeclareMathOperator{\inj}{inj}
\DeclareMathOperator{\Id}{1}
\DeclareMathOperator{\Ker}{Ker}

\DeclareMathOperator{\Imm}{Im}
\DeclareMathOperator{\Hom}{Hom}
\DeclareMathOperator{\cHom}{\mathcal{H}om}

\DeclareMathOperator{\Cone}{Cone}
\DeclareMathOperator{\CCone}{Cocone}
\DeclareMathOperator{\EE}{\mathbb{E}}
\DeclareMathOperator{\add}{add}

\DeclareMathOperator{\thi}{thick}
\DeclareMathOperator{\Filt}{Filt}

\def\K{\ensuremath{\mathcal{K} }}

\newcommand{\X}{\mathcal{X}}
\newcommand{\Y}{\mathcal{Y}}
\newcommand{\Tt}{\mathscr{T}}
\newcommand{\Ww}{\mathscr{W}}
\newcommand{\T}{\mathcal{T}}
\newcommand{\F}{\mathcal{F}}
\newcommand{\U}{\mathcal{U}}
\newcommand{\W}{\mathcal{W}}
\newcommand{\C}{\mathcal{C}}
\newcommand{\D}{\mathcal{D}}
\newcommand{\Z}{\mathcal{Z}}
\newcommand{\HH}{\mathcal{H}}
\newcommand{\J}{\mathcal{J}}
\newcommand{\Ss}{\mathcal{S}}

\newcommand{\Kb}{\mathcal{K}^{b}(\proj \Lambda)}
\newcommand{\Kproj}{\mathcal{K}^{[-1,0]}(\proj \Lambda)}
\newcommand{\scon}[5]{\begin{tikzcd}[column sep = small, ampersand replacement=\&] #1 \arrow[r, rightarrowtail, "#2"] \& #3 \arrow[r, twoheadrightarrow, "#4"] \& #5 \end{tikzcd}}
\newcommand{\con}[6]{\begin{tikzcd}[column sep = small, ampersand replacement=\&] #1 \arrow[r, rightarrowtail, "#2"] \& #3 \arrow[r, twoheadrightarrow, "#4"] \& #5 \arrow[r, dashrightarrow, "#6"] \& {\!} \end{tikzcd}}
\newcommand{\defl}[3]{\begin{tikzcd}[column sep = small, ampersand replacement=\&] #1 \arrow[r, twoheadrightarrow, "#2"] \& #3 \end{tikzcd}}
\newcommand{\infl}[3]{\begin{tikzcd}[column sep = small, ampersand replacement=\&] #1 \arrow[r, rightarrowtail, "#2"] \& #3 \end{tikzcd}}
\newcommand{\complex}[3]{\begin{tikzcd}[sep=small, cramped, ampersand replacement=\&] #1 \arrow[d, "#2"] \\ #3 \end{tikzcd}}

\DeclareMathOperator{\ctor}{cotor}
\DeclareMathOperator{\cctor}{c-cotor}
\DeclareMathOperator{\silt}{silt}
\DeclareMathOperator{\tautilt}{s\tau-tilt}

\DeclareMathOperator{\wide}{wide}

\DeclareMathOperator{\fwide}{l-wide}

\DeclareMathOperator{\ftor}{f-tors}
\DeclareMathOperator{\tor}{tors}

\DeclareMathOperator{\fthi}{inj-thick}
\DeclareMathOperator{\smc}{smc}
\DeclareMathOperator{\tsmc}{2-smc}
\DeclareMathOperator{\tsilt}{2-silt}

\usepackage{array}
\usetikzlibrary{calc}
\usetikzlibrary{shapes}
\usepackage{tabularx}
\usepackage{longtable}
\pgfdeclarelayer{bg}
\pgfsetlayers{bg,main}

\begin{document}
\title{On \lowercase{g}-finiteness in the category of projective presentations}
\author{Monica Garcia}
\address{Laboratoire de Mathématiques de Versailles, UVSQ, CNRS, Université Paris-Saclay. Office 3305, Bâtiment Fermat, 45 Avenue des États-Unis, 78035 Versailles CEDEX, France}
\email{monica.garcia@uvsq.fr}
%
\begin{abstract}
	
	\smallskip
	
	We provide new equivalent conditions for an algebra $\Lambda$ to be $g$-finite, analogous to those established by L.~Demonet, O.~Iyama, and G.~Jasso, but within the category of projective presentations $\Kproj$. We show that an algebra has finitely many isomorphism classes of basic $2$-term silting objects if and only if all cotorsion pairs in $\Kproj$ are complete. Furthermore, we establish that this criterion is also equivalent to all thick subcategories in $\Kproj$ having enough injective and projective objects. 
\end{abstract}

\maketitle

\tableofcontents
	
\section{Introduction}

The notion of support $\tau$-tilting module \cite{adachi2014tilting} was inspired in part by the additive categorification of cluster algebras (see \cite{amiot2009cluster, buan2006tilting} or survey papers \cite{keller2019algebres, amiot2011generalized}). When $\Lambda$ is the Jacobi algebra $\mathcal{J}(Q,W)$ associated with a quiver $Q$ with non-degenerate potential $W$, the set of reachable basic 2-term silting complexes is in bijection with the set of clusters of the cluster algebra $\mathcal{A}_Q$ associated to $(Q,W)$. One of the first questions to be settled when cluster algebras were introduced is whether the set of clusters of a given cluster algebra $\mathcal{A}_Q$ is finite. This turned out to be equivalent to $Q$ being mutation-equivalent to a quiver of Dynkin type \cite{fomin2003finite}. Since Dynkin quivers are representation finite, $\K^{[-1,0]}(\proj J(Q,W))$ has only finitely many isoclasses of indecomposable objects and thus finitely many isoclasses of basic $2$-term silting objects. For a general finite-dimensional $\Bbbk$-algebra $\Lambda$, this does not have to be the case, prompting the introduction of the following definition.
 
\begin{definition}\cite{demonet2019tilting}
	Let $\Lambda$ be a finite-dimensional $\Bbbk$-algebra. We say that $\Lambda$ is \textit{$g$-finite} if it admits only finitely many isomorphism classes of basic $2$-term silting objects. 
\end{definition}

The study of $g$-finite algebras was introduced by L.~Demonet, O.~Iyama and G.~Jasso in \cite{demonet2019tilting}, who showed that an algebra $\Lambda$ being $g$-finite has deep implications on the structure of $\Mod \Lambda$. 

\begin{theorem}\cite[Theorems 3.8 and 4.2]{demonet2019tilting}\label{demonet_finite} Let $\Lambda$ be a finite-dimensional $\Bbbk$-algebra. The following are equivalent: 
	\begin{enumerate}
		\item $\Lambda$ is $g$-finite. 
		\item There exist finitely many functorially finite torsion classes in $\Mod \Lambda$.
		\item All torsion classes in $\Mod \Lambda$ are functorially finite.
		\item There exist finitely many isomorphism classes of bricks in $\Mod \Lambda$.  
	\end{enumerate}
\end{theorem}

Functorially finite torsion classes are in one-to-one correspondence with \textit{left finite wide subcategories} \cite{marks2017torsion, ingalls2009noncrossing}. Since left finite wide subcategories are defined as those such the smallest torsion class containing them is functorially finite, the following corollary is immediate. 

\begin{corollary}\label{coro_wide_lf}
	 Let $\Lambda$ be a finite-dimensional $\Bbbk$-algebra. If $\Lambda$ is $g$-finite, then all wide subcategories are left finite. 
\end{corollary}

In this article, we study the consequences that being $g$-finite has on the correspondences between complete cotorsion pairs, $2$-term silting objects and thick subcategories with enough injectives in the extriangulated category $\Kproj$ introduced in \cite{garcia2023thick}. Specifically, we provide analogs of \cref{demonet_finite} and \cref{coro_wide_lf}. 

\begin{theorem}\label{monica_finite}
	Let $\Lambda$ be a finite-dimensional $\Bbbk$-algebra. The following are equivalent: 
	\begin{enumerate}
		\item $\Lambda$ is $g$-finite.
		\item There exist finitely many complete cotorsion pairs in $\Kproj$.
		\item All cotorsion pairs in $\Kproj$ are complete. 
		\item There exist finitely many thick subcategories in $\Kproj$.
	\end{enumerate}
\end{theorem}

\begin{theorem*}[\textbf{\ref{allthick}}]
	Suppose $\Lambda$ is $g$-finite. Then all thick subcategories of $\Kproj$ have enough injectives. 
\end{theorem*}

To show condition $\mathit{(3)}$ of \cref{monica_finite} we extend a result of D.~Pauksztello and A.~Zvonareva \cite{pauksztello2023cotorsion}. 

\begin{theorem*}[\textbf{\ref{cotortorbij}}]
	Let $\Lambda$ be a finite-dimensional $\Bbbk$-algebra. Then the functor $H^0 : \Kproj \rightarrow \Mod \Lambda$ induces a bijection 
	\begin{align*}
		H^0  : \ctor \Lambda &\rightarrow \tor \Lambda \\
		(\X, \Y) &\mapsto H^0(\Y) . 
	\end{align*}
\end{theorem*}

Recall that the set of torsion classes of $\Mod \Lambda$ equipped with the inclusion forms a lattice, that is, a poset in which any two elements have an unique least upper bound (or \textit{join}), and an unique greatest lower bound (or \textit{meet}) (see for instance \cite{thomas2021intro}). In a similar way, we can equip the set of cotorsion pairs in $\Kproj$ with a poset structure: for any $(\X, \Y)$ and $(\X', \Y') \in \ctor \Lambda$ we say that $(\X, \Y) \leq (\X', \Y')$ if $\Y \subset \Y'$. Since $H^0$ preserves inclusions, \cref{cotortorbij} gives the following corollary. 
\begin{corollary} Let $\Lambda$ be a finite-dimensional $\Bbbk$-algebra. Then the set $\ctor \Lambda$ of cotorsion pairs in $\Kproj$ forms a lattice which is isomorphic to the lattice $\tor \Lambda$ of torsion classes in $\Mod \Lambda$. 
\end{corollary}

To establish the equivalence between being $g$-finite and condition $\mathit{(4)}$, we show that when $\Lambda$ is $g$-finite, then all thick subcategories in $\Kproj$ contain a presilting object. The proof of this fact relies on geometric results on the degeneration of objects of the triangulated category $\K^b(\proj \Lambda)$ \cite{bernt2005degenerations}. We also develop a reduction argument by explicitly computing the silting reduction of the extriangulated category $\Kproj$ by a presilting object in the manner of that developed by O.~Iyama and D.~Yang for triangulated categories in \cite{iyama2018silting}. Moreover, we show that the categories obtained after reduction are equivalent to the extriagnulated category $\per^{[-1, 0]}(\Gamma)$ of $2$-term complexes over a certain non positive dg-algebra $\Gamma$. 

\section*{Acknowledgments}
I am grateful to P.-G. Plamondon for his insights and support throughout the development of my Ph.D. thesis, from which this article originates. I would also like to thank E. D. B{\o}rve for the many interesting discussions and for sharing with me a preliminary version of \cite{borve2024silting}. I also thank A. Shah for his comments, which led to the addition of \cref{thekey1}. Finally, I thank the organizers of the Oberwolfach Workshop ``Cluster Algebras and Its Applications", where this work was first presented. 

\section{Preliminaries}
The goal of this section is to introduce the necessary preliminaries for Sections \ref{sec_reduction} and \ref{sec_thickg}. We recall the definition of the extriangulated category of $2$-term perfect complexes $\per^{[-1,0]}(\Gamma)$ over a dg algebra $\Gamma$, of which $\Kproj$ is a special case. For a broader introduction to dg categories see \cite{keller1994DG, keller2006dg}. Most of the results in this section are taken from \cite{brustle2013ordered, koening2014silting}. Throughout this paper, we fix a field $\Bbbk$. 
\subsection{dg algebras and dg categories}
Recall that a \textit{differential graded algebra}, or dg algebra for short, is a graded $\Bbbk$-algebra $\Gamma = \bigoplus_{i \in \mathbb{Z}} \Gamma^i$ equipped with a homogeneous $\Bbbk$-linear map of degree one $d_\Gamma : \Gamma \rightarrow \Gamma$ such that 
	\begin{enumerate}[i)]
		\item $d_\Gamma(ab) = d_\Gamma(a)b + (-1)^i a d_\Gamma(b) \ $  $\forall a \in \Gamma ^i\ \text{and } b \in \Gamma$;
		\item $d_\Gamma \circ d_\Gamma = 0$.
	\end{enumerate}
	We call $d_\Gamma$ the \textit{differential} of $\Gamma$. We say a a dg algebra $\Gamma$ is \textit{non-positive} if $\Gamma^i = 0$ for all $i \geq 1$. We say that $\Gamma$ it is \textit{finite-dimensional} if it is finite-dimensional as a $\Bbbk$-vector space. 
A dg $\Gamma$-module is a graded right $\Gamma$-module $X = \bigoplus_{i \in \mathbb{Z}} X^i$ equipped with a homogeneous $\Bbbk$-linear map of degree one $d_X: X \rightarrow X$ such that 
\[d_X(xa) = d_X(x)a + (-1)^i x \,d_\Gamma(a) \ \forall a \in \Gamma \ \text{and } x \in X^i.\]
For a given dg $\Gamma$-module $X$, we denote by $H^i(X) = \Ker d^i_X/\Imm d^{i-1}_X$. Recall that for any $X$ and $Y$ dg $\Gamma$-modules we have the dg $\Bbbk$-module
	\[ \cHom_\Gamma(X, Y) = \bigoplus_{i \in \mathbb{Z}} \cHom^i_\Gamma(X, Y)\]
	where $\cHom^i(X, Y)$ is the subset of $\displaystyle\prod_{j \in \mathbb{Z}} \Hom_\Gamma(X^j, Y^{j+i})$ of elements $(f_j)_{j \in \mathbb{Z}}$ such that 
	\[f_j(x)a = f_{j+n}(xa) \text{ for all } x \in X^j \text{ and } a \in \Gamma ^n;\]
	whose differential is given by the map 
	\[ f \in \cHom^i_\Gamma(X, Y) \mapsto f \circ d_X - (-1)^id_Y \circ f.\]
When $X = Y$, we will write $\cHom_\Gamma(X, Y) = \EEn_\Gamma(X)$. Note that given two dg $\Gamma$-modules $X$ and $Y$, the kernel of $d^0$, which we denote by $Z^0(\cHom_\Gamma(X,Y))$, is nothing but the set of $\Gamma$-linear maps that commute with the differentials of $X$ and $Y$. For any dg $\Gamma$-module $X$ with differential $d_X$ and $i \in \mathbb{Z}$, let $X[i]$ be the dg module whose underlying graded $\Gamma$-module is $\bigoplus_{j \in \mathbb{Z}} X^{i+j}$ equipped with the differential $(-1)^id_X$, then the sets $H^i(\cHom_\Gamma(X,Y))$ correspond to the sets $Z^0(\cHom_\Gamma(X,Y[i]))$ modulo the homotopy relation. 

\begin{definition} Let $\Gamma$ be a dg algebra. We denote by $\C(\Gamma)$ the category whose objects are dg $\Gamma$-modules and whose morphism spaces are given by
	\[\Hom_{\C(\Gamma)}(X, Y) = Z^0(\cHom_\Gamma(X, Y)).\]
We will denote by $\K(\Gamma)$ the category whose objects are the same as those of $\C(\Gamma)$ but whose morphism spaces are given by
\[\Hom_{\K(\Gamma)}(X, Y) = H^0(\cHom_\Gamma(X, Y)).\]
The \textit{derived category} associated to $\Gamma$ is the triangulated quotient
\[\D(\Gamma) = \K(\Gamma) / \mathcal{N},\]
where $\mathcal{N}$ is the thick subcategory of objects whose cohomologies are all $0$. In other words, $\D(\Gamma)$ is the Verdier localisation of $\K(\Gamma)$ by quasi-isomorphisms.
\end{definition}

We will denote by $\D_{fd}(\Gamma)$ the full subcategory of $\D(\Gamma)$ of dg $\Gamma$-modules whose total cohomology is finite-dimensional. We will consider as well the category $\per(\Gamma) = \thi_{\D(\Gamma)}(\Gamma)$ the smallest triangulated full subcategory of $\D(\Gamma)$ containing $\Gamma$ and closed under direct summands. The subcategory  $\per(\Gamma)$ is known as the \textit{category of perfect complexes} over $\Gamma$. 

\begin{example}
	Let $\Lambda$ be a finite-dimensional $\Bbbk$-algebra. When $\Gamma$ is seen as a dg algebra concentrated in degree $0$, $\C(\Lambda)$ is precisely the category $\C(\text{Mod } \Lambda)$ of complexes of (not necessarily finite-dimensional) $\Lambda$-modules, $\D(\Lambda)$ is $\D(\text{Mod } \Lambda)$, $\D_{fd}(\Lambda)$ corresponds to $\D^b(\Mod \Lambda)$ and $\per(\Lambda)$ is equivalent to $\Kb$. 
\end{example}
	
\begin{theorem}\cite[Lemma 4.1]{koening2014silting} \cite[Theorem 3.8 b)]{keller2006dg}\label{kellermorita} Let $\C$ be an algebraic triangulated category, that is, equivalent to the stable category of a Frobenius category. Suppose that $\C$ has split idempotents and a silting object $X \in \C$. Then there exists a non-positive dg algebra $\Gamma$ and a triangle equivalence 
	\[\per(\Gamma) \xlongrightarrow{\simeq} \C\]
which takes $\Gamma$ to $X$. In particular $H^0(\Gamma) \simeq \Hom_{\per(\Gamma)}(\Gamma, \Gamma) \simeq \Hom_{\C}(X,X)$. 
\end{theorem}
\subsection{$2$-term silting objects and $2$-term simple-minded collections}

Let $\D$ be a triangulated category which is essentially
small, Krull–Schmidt, $\Bbbk$-linear and Hom-finite with shift functor $\Sigma$. Recall that a silting object $V$ of $\D$ is an object that satisfying $\Hom_{\D}(V, \Sigma ^i V) = 0$ for all $i > 0$ and such that $\thi_{\D}(V) = \D$.  Let $V$ be a silting object in $\D$, we say that an object $X \in \D$ is \textit{$2$-term} with respect to $V$ if there exist $Y, Y' \in \add(V)$ and a triangle 
\[Y' \rightarrow Y \rightarrow X \dashrightarrow Y'[1]. \] 
We will denote by $V * \Sigma V$ the full subcategory of $2$-term objects with respect to $V$. When $\D = \per(\Gamma)$, where $\Gamma$ is a finite-dimensional non-positive dg algebra, then $\Gamma$ is itself a silting object in $\per(\Gamma)$, since $\Hom_{\per(\Gamma)}(\Gamma, \Gamma[i]) \simeq H^i(\Gamma) = 0$ for all $i >0$. We write $\per^{[-1, 0]}(\Gamma) := \Gamma * \Gamma[1]$ for the full subcategory of $2$-term objects of $\per \Gamma$ with respect to $\Gamma$. We will refer to $\per^{[-1, 0]}(\Gamma)$ simply as the category of $2$-term objects of $\Gamma$ and denote by $\tsilt \Gamma$ the set isomorphism classes of basic $2$-term silting objects in $\per(\Gamma)$. 

\begin{lemma}[Bongartz Completion]\cite[Lemma 4.2]{iyama2014twoterm}\label{Bconp_silting}
	Let $\D$ be a triangulated category which is essentially
	small, Krull–Schmidt, $\Bbbk$-linear and Hom-finite with shift functor $\Sigma$. Suppose that $\D$ possesses a silting object $V \in \D$. Let $U \in \D$ be a presilting object in $V *\Sigma V$, then there exists and object $U' \in V * \Sigma V$ such that $U \oplus U'$ is a silting object in $\D$. 
\end{lemma}

The following consequence of \cref{Bconp_silting} is of particular importance. It allows us to ascertain whether a $2$-term presilting object $U$ in a triangulated category $\D$ is silting by examining its number of non-isomorphic direct summands, which we denote by $|U|$. 

\begin{proposition}\cite[Proposition 4.3]{iyama2014twoterm}\label{maximal_silt} Let $\D$ be as in \cref{Bconp_silting} and let $U$ be a presilting object in $V *\Sigma V$. Then $U$ is silting if and only if $|V| = |U|$. 
\end{proposition}

\begin{definition}\cite{rickard2002equivalences}
	Let $\C$ be a triangulated category with shift functor $\Sigma$. A collection of objects $X_1, X_2, \cdots X_r$ is said to be \textit{simple-minded} if the following conditions hold for any $1 \leq i, j \leq n$:
	\begin{enumerate}[i)]
		\item $\Hom_\C(X_i, \Sigma^m X_j) = 0$ for all $m < 0$, 
		\item $\En_\C(X_i)$ is a division algebra and $\Hom_\C(X_i, X_j) = 0$ when $i \neq j$. 
		\item $\thi_\C(\{X_1, \cdots X_r\}) = \C$.
	\end{enumerate}
We denote by $\smc \Gamma$ the set of isomorphism classes of simple-minded collections of $\D_{fd}(\Gamma)$. 
\end{definition}
If $\Gamma$ is a finite-dimensional non-positive dg algebra, then the set $\{S_1, \cdots S_n\}$ of pairwise non-isomorphic simple $H^0(\Gamma)$-modules is a simple-minded collection in $\D_{fd}(\Gamma)$ \cite[Appendix A.1]{brustle2013ordered}. We say that a simple-minded collection of $\D_{fd}(\Gamma)$ is \textit{$2$-term} if $H^j(X_i) = 0$ for all $j \neq 0, -1$ and all $1\leq i \leq n$. We denote by $\tsmc \Gamma$ the set of isomorphism classes of $2$-term simple-minded collections. 

\begin{remark}\cite[Remark 4.11]{brustle2013ordered}
	Let $\Lambda$ be a finite-dimensional $\Bbbk$-algebra. Suppose that $\{X_1, \cdots X_n\}$ is a $2$-term simple-minded collection in $\D^b(\Mod \Lambda)$, then for any $1 \leq i \leq n$ the object $X_i$ belongs to either $\Mod \Lambda$ or $(\Mod \Lambda)[1]$. 
\end{remark}

\begin{theorem}\cite{keller2013weight, koening2014silting} \cite[Corollary 4.1]{brustle2013ordered}\label{silt_smc}
	Let $\Gamma$ be a homologically smooth non-positive dg algebra or a finite-dimensional $\Bbbk$-algebra. Then there exists a bijection 
	\[ \Omega : \silt \Gamma \rightarrow \smc \Gamma\]
	that restricts to a bijection $\Omega : \tsilt \Gamma \rightarrow \tsmc \Gamma$ between the set of $2$-term silting objects and $2$-term simple-minded collections. 
\end{theorem}

When $\Gamma = \Lambda$ is a finite-dimensional $\Bbbk$-algebra, the bijection in \cref{silt_smc} is given in the following way. Let $U$ be a silting object in $\per \Lambda$. By \cref{kellermorita}, there exists a non-positive dg algebra $B$ together with a triangle equivalence $\D(B) \rightarrow \D(\Lambda)$ that takes $B$ to $U$. The simple-minded collection $\{X_1, \cdots X_n\}$ corresponding to $U$ under the map $\Omega$ is the image under the equivalence $\D(B) \rightarrow \D(\Lambda)$ of a complete collection of non-isomorphic simple $H^0(B)$-modules. In particular, any simple-minded collection has $n = |U| = |\Lambda|$ elements.  

\subsection{0-Asulander extriangulated categories} 
Extriangulated categories were introduced by H.~Nakaoka and Y.~Palu in \cite{nakaoka2019extriangulated} as a mean to generalize both triangulated and exact categories. Examples of extriangulated categories include extension-closed full subcategories of an extriangulated category. In particular, for any given dg algebra $\Gamma$, the category $\per^{[-1,0]}(\Gamma)$ of $2$-term complexes is extriangulated when equipped with the extension functor given by $\EE(X,Y) = \Hom_{\per(\Lambda)}(X, Y[1])$ for any $X, Y \in \per^{[-1,0]}(\Gamma)$. 

We say that a sequence $\scon{X}{f}{Y}{g}{Z}$ in $\per^{[-1,0]}(\Gamma)$ is a \textit{conflation} if there exists $\delta \in \Hom_{\per(\Lambda)}(Z, X[1])$ such that $(f,g,\delta)$ is a triangle in $\per(\Gamma)$. In such cases, we say that $f$ is an \textit{inflation}, which we denote by $\infl{X}{f}{Y}$, and we write that $Z = \Cone(f)$. Similarly, we say that $g$ is a \textit{deflation}, which we denote by $\defl{Y}{g}{Z}$, and we write $X = \CCone(g)$. The following definition is due to M.~Gorsky, H.~Nakaoka and Y.~Palu.

\begin{definition}\cite[Definition 3.7]{gorsky2023hereditary} An extriangulated category is a \textit{0-Auslander} if it has enough projectives, enough injectives, global dimension at most one, dominant dimension at least one, and codominant dimension at least one.
\end{definition}

When $\Gamma$ is a non-positive dg algebra, $\per^{[-1, 0]}(\Gamma)$ is an example of a 0-Auslander extriangulated category. Indeed, by definition for any $X \in \per^{[-1, 0]}(\Gamma)$ there exists a conflation $\scon{U'}{}{U}{}{X}$ where $U, U' \in \add(\Gamma)$ are projective objects in $\per^{[-1, 0]}$. Similarly, there is a conflation $\scon{X}{}{V}{}{V'}$, where $V, V' \in \add(\Gamma[1])$ are injective objects. To see that $\per^{[-1,0]}(\proj \Lambda)$ has dominant and codominant dimension at least one (\cite[Definition 3.5]{gorsky2023hereditary}), remark that all projective objects in $\per^{[-1, 0]}(\Gamma)$ lie in $\add(\Gamma)$, all injective objects lie in $\add(\Gamma[1])$ and for all $U \in \add(\Gamma)$ we have a conflation $\scon{U}{}{0}{}{U[1]}$. 

\begin{proposition}\cite[Proposition 3.2, Corollary 3.3]{fang2023extriangulated} \label{idealquotient} Let $\C$ be a 0-Auslander extriangulated category and let $J$ be an ideal generated by morphisms with injective domain and projective codomain. Then the ideal quotient $\C/J$ is a $0$-Auslander extriangulated category. 
\end{proposition}

We now recall the relation between $2$-term silting objects of a finite-dimensional non-positive dg algebra $\Gamma$ and those of the finite-dimensional algebra $H^0(\Gamma)$. Let $\bar{\Gamma} = H^0(\Gamma)$ and $p : \Gamma \rightarrow \bar{\Gamma}$ the canonical projection. The map $p$ gives rise to the triangulated functor
\begin{align*}
	 p_* : \per(\Gamma) &\longrightarrow \per(\bar{\Gamma})\\
	 X &\longmapsto X \otimes_{\bar{\Gamma}}\bar{\Gamma}, 		
\end{align*}  
which we refer to as the \textit{induction functor}. Since $\Hom_{\per(\Gamma)}(\Gamma, \Gamma) = \bar{\Gamma} = \Hom_{\bar{\Gamma}}(\bar{\Gamma}, \bar{\Gamma})$, then $p_*$ induces an equivalence $\add_{\per(\Gamma)}(\Gamma) \simeq \add_{\per(\bar{\Gamma})}(\bar{\Gamma})$, where $\per(\bar{\Gamma}) \simeq \K^b(\proj \bar{\Gamma})$. 

\begin{proposition}\cite[Proposition A.5]{brustle2013ordered}\label{indc_funct} Let $\mathcal{I}$ be the ideal of $\per^{[-1,0]}(\Gamma)$ consisting of morphisms factoring through morphisms $X[1]\rightarrow Y$ with $X, Y \in \add_{\per(\Gamma)}(\Gamma)$. Then $\mathcal{I}^2 = 0$ and $p_*$ induces an equivalence of $\Bbbk$-linear categories $\per^{[-1,0]}(\Gamma)/\mathcal{I} \rightarrow \K^{[-1,0]}(\proj H^0(\Gamma))$. In particular, $p_*$ is full, detects isomorphisms, preserves indecomposability and induces a bijection between isomorphism classes of objects of $\per^{[-1,0]}(\Gamma)$ and $\K^{[-1,0]}(\proj H^0(\Gamma))$. 
\end{proposition}

\begin{remark}
	The previous proposition and \cref{idealquotient}, imply that the induction functor $p_*$ is an extriangulated functor and that the equivalence in \cref{indc_funct} induces an equivalence of $0$-Auslander extriangulated categories. Indeed, recall than an extriangulated functor between two extriangulated categories $\C$ and $\C'$ is given by an additive functor $\F : \C \rightarrow \C'$ and a natural transformation $\alpha : \EE_{\C} \Rightarrow \EE_{\C'}\circ(\F^{op}\times \F)$ such that for any $X, Z \in \C$ and $\delta \in \EE_{\C}(Z, X)$ realized by a conflation $\con{X}{f}{Y}{g}{Z}{\delta}$, then $\alpha(\delta) \in \EE_{\C'}(\F(Z), \F(X))$ is realized by the conflation $\con{\F(X)}{\F(f)}{\F(Y)}{\F(g)}{\F(Z)}{\alpha(\delta)}$ (\cite[Defintion 2.32]{shah2021transport}). Moreover, an extriangulated functor $(\F, \alpha)$ is an extriangulated equivalence if and only if $\F$ is an additive equivalence and $\alpha$ is a natural isomorphism (\cite[Proposition 2.13]{nakaoka2022localization}). In our case, the induction functor $p_*$ is an additive equivalence between $\per^{[-1,0]}(\Gamma)$ and $\K^{[-1,0]}(\proj H^0(\Gamma))$ by \cref{indc_funct}. The corresponding natural transformation is the one making $p_*: \per(\Gamma) \rightarrow K^b(\proj H^0(\Gamma))$ into a triangulated (and hence extriangulated \cite[Theorem 2.33]{shah2021transport}) functor. By \cite[Theorem 2.8]{fang2023extriangulated} for any $X, Z \in \per^{[-1,0]}(\Gamma)$ we have that 
	\[\EE_{\per^{[-1,0]}(\Gamma)/\mathcal{I}}(Z,X) = \EE_{\per^{[-1,0]}(\Gamma)}(Z,X),\]
	and that conflations in $\per^{[-1,0]}(\Gamma)/\mathcal{I}$ are precisely the image of those in $\per^{[-1,0]}(\Gamma)$. Thus, the only thing left to verify is that for any $X, Z \in \per^{[-1,0]}(\Gamma)$ the natural transformation associated to $p_*$ induces an isomorphism between \[ \Hom_{\per(\Gamma)}(Z, X[1]) \simeq \Hom_{K^b(\proj H^0(\Gamma))}(p_*(Z), p_*(X)[1]). \] This follows essentially from \cite[Proposition A.4]{brustle2013ordered}. We include a proof for the convenience of the reader. 
\end{remark}

\begin{proposition}\cite[Proposition A.4]{brustle2013ordered}\label{thekey1}
	Let $X, Z \in \per^{[-1,0]}(\Gamma)$. Then the functor $p_*$ induces an isomorphism 
	\[ \Hom_{\per(\Gamma)}(Z, X[1]) \simeq \Hom_{K^b(\proj H^0(\Gamma))}(p_*(Z), p_*(X)[1]). \]
\end{proposition}

\begin{proof}
	Let $X, Z \in \per^{[-1,0]}(\Gamma)$. There are conflations
	\begin{gather}\label{conf1}
		\scon{X'}{}{X}{}{X''[1]}
	\end{gather}
	\begin{gather}\label{conf2}
		\scon{Z'}{}{Z}{}{Z''[1]}
	\end{gather} with $X', X'', Z', Z'' \in \add_{\per(\Gamma)}(\Gamma)$. By applying the functor $\Hom_{\per(\Gamma)}(-, X''[1])$ and $\Hom_{K^b(\proj H^0(\Gamma))}(-, p_*(X'')[1])$ to the conflation \ref{conf2} and its image under $p_*$, we get the following commutative diagram with exact rows
	
\begin{center}
	\makebox[\textwidth][c]{\includegraphics[width=1.4\textwidth]{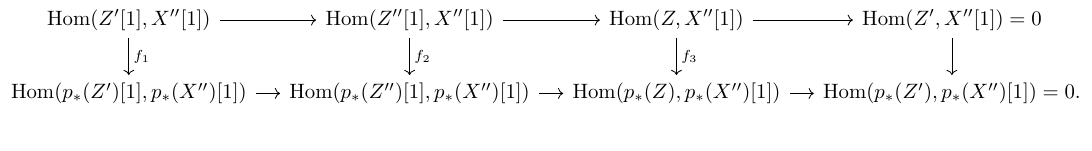}}
\end{center}
\vspace*{-2mm}

Since $p_*$ is an additive equivalence between $\add_{\per(\Gamma)}(\Gamma)$ and $\add_{\per(\bar{\Gamma})}(\bar{\Gamma})$, both $f_1$ and $f_2$ are isomorphisms, and by the Four Lemma then so if $f_3$. By applying $\Hom_{\per(\Gamma)}(-, X'[1])$ and $\Hom_{K^b(\proj H^0(\Gamma))}(-, p_*(X')[1])$ and using a similar argument, we get that $p_*$ induces an isomorphism $\Hom(Z, X'[1]) \xrightarrow{g_3} \Hom(p_*(Z), p_*(X')[1])$. Now we apply $\Hom_{\per(\Gamma)}(Z, -)$ and $\Hom_{K^b(\proj H^0(\Gamma))}(p_*(Z), -)$ to the conflation \ref{conf1} and its image under $p_*$, which produces de following commutative diagram with exact rows 

\begin{center}
	\makebox[\textwidth][c]{\includegraphics[width=1.3\textwidth]{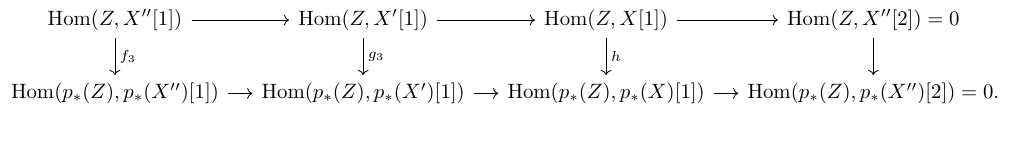}}
\end{center}

Since $f_3$ and $g_3$ are isomorphisms, by the Four Lemma then so is $h$. 

\end{proof}

\begin{corollary}\label{thekey}
	Let $X, Z \in \pero(\Gamma)$ and suppose there is a conflation
	\begin{equation}\label{conflation}
		\con{p_*(X)}{}{\bar{Y}}{}{p_*(Z)}{f}
	\end{equation} 
	where $\bar{Y} \in \K^{[-1,0]}(\proj H^0(\Gamma))$. Then there exists $Y \in \pero(\Gamma)$ and a conflation $$\con{X}{}{Y}{}{Z}{F}$$ whose image by $p_*$ is the conflation \ref{conflation}. 
\end{corollary}
	
\begin{proof}
 Let $f \in \Hom_{\K^{b}(\proj H^0(\Gamma))}(p_*(Z), p_*(X)[1])$ and $\con{p_*(X)}{}{\bar{Y}}{}{p_*(Z)}{f}$ be a conflation realizing $f$. By \cref{thekey1} there exist $F \in \Hom_{\per(\Gamma)}(Z, X[1])$ such that $p_*(F) = f$. By letting $Y = \CCone(F)$ we get the result.
\end{proof}

We end this section with the following result by T.~Brüstle and D.~Yang, which relates the set of isomorphism classes of $2$-term silting objects in $\per^{[-1,0]}(\Gamma)$ and that of $\per^{[-1,0]}(\bar{\Gamma})$. 
\begin{proposition}\cite[Proposition A.3]{brustle2013ordered}\label{ind_silting}
	The induction functor $p_*: \per(\Gamma) \rightarrow \K^b(\proj \bar{\Gamma})$ induces a bijection between the sets of isomorphism classes of $2$-term silting objects $\tsilt \Gamma$ and $\tsilt \bar{\Gamma}$. 
\end{proposition}

\section{Silting reduction in $\Kproj$}\label{sec_reduction}

The goal of this section is to explicitly describe the reduction of $\Kproj$ with respect to a presilting object $U$. Reductions for hereditary extriangulated categories were treated in general in \cite[Section 2.2.3]{gorsky2023hereditary}. In our setting, we show that the reduction of $\Kproj$ by a $2$-term presilting complex is equivalent to the $2$-term category of perfect complexes over a non-positive dg algebra. This is done using Iyama-Yoshino reduction in $\Kb$ \cite{yoshino2008mutation} as well as Iyama-Yang's results showing that the Verdier localisation by a presilting object is a reduction \cite{iyama2018silting}. We show that both operations are compatible with those in \cite{gorsky2023hereditary}. This is a particular case of the reduction of a $0$-Auslander triangulated category $\K$ with respect to a presilting object $U$ being equivalent to the Verdier localization $\K/\thi(U)$, which was shown in general in \cite{borve2024silting}.

\subsection{Thick subcategories generated by $2$-term presilting complexes}

Let $\HH$ be a full subcategory of an extriangulated category $\K$. We say that $\HH$ is \textit{closed under extensions} if for every conflation $\scon{X}{f}{Y}{g}{Z}$ such that $X, Z \in \HH$, then $Y \in \HH$ as well. We say that $\HH$ is \textit{closed under cones} (resp. \textit{closed under cocones}) if for any inflation $\infl{X}{f}{Y}$ (resp. deflation $\defl{X}{g}{Y}$) with $X, Y \in \HH$ we have that $\Cone(f) \in \HH$ (resp. $\CCone(g) \in \HH$). 

\begin{definition}\cite{nakaoka2022localization} Let $\K$ be an extriangulated category. We say that a full subcategory $\HH \subset \K$ is \textit{thick} if it is stable under finite direct sums and direct summands, and closed under extensions, cones and cocones. 
	
\end{definition}

If $\D$ is a triangulated category with shift functor $\Sigma$, then a full subcategory $\T$ of $\D$ is thick (with respect to the previous definition) if and only if it is a triangulated subcategory which is stable under direct summands. Recall that if $\K \subset \D$ is a full and closed under extensions, then it is extriangulated \cite{nakaoka2019extriangulated} with respect to $\Sigma|_{\K}$. The next proposition follows from the definitions. 
\begin{proposition}\label{interthick}
	Let $\T \subset \D$ be a thick subcategory of $\D$ and let $\K \subset \D$ be closed under extensions and direct summands. Then $\T \cap \K$ is a thick subcategory of the extriangulated category $\K$. 
\end{proposition}

In what follows, $\D = \Kb$ and $\K = \Kproj$. Since we will work with the notion of thick subcategory in the triangulated category $\Kb$ as well as the notion of thick subcategory in the extriangulated category $\Kproj$, to avoid confusion, for any subcategories $\C \in \Kb$ and $\HH \in \Kproj$ we will denote $\thi_b(\C)$ the smallest (triangulated) thick subcategory in $\Kb$ containing $\C$, and by $\thi_{[-1,0]}(\HH)$ the smallest (extringulated) thick subcategory  of $\Kproj$ that contains $\HH$. The following proposition relates both notions. 

\begin{lemma}\label{explicitthick}
	Let $U \in \Kproj$ be a $2$-term presilting object and $\U=\add(U)$, then 
	\[ \thi_{[-1,0]}(U)= \thi_b(U) \cap \Kproj = \left(\U[-1] * \U * \U[1]\right) \cap \Kproj. \]
\end{lemma}

In order to prove \cref{explicitthick} we will make use of the following proposition. 

\begin{proposition} \label{thickdes} Let $\D$ be a Hom-finite, Krull-Schmidt triangulated category. If $\U = \add(\U)$ is a presilting subcategory of $\C$, then 
	\begin{enumerate}[(i)]
		\item \cite[Propositions 2.7]{iyama2018silting} For all $n \leq 0$, $$\U  * \U[1] * \cdots *\U[n] = \add(\U  * \U[1] * \cdots *\U[n]) .$$
		\item \cite[Propositions 2.15]{aihara2012silting}
		\[ \begin{split}
			\thi_\D(\U)   &  = \bigcup_{n \geq 0} \add\left(\U[-n] * \U[1-n]* \cdots * \U[n-1] * \U[n]\right) \\
			& = \bigcup_{n \geq 0} \U[-n] * \U[1-n]* \cdots * \U[n-1] * \U[n].
		\end{split} \]
	\end{enumerate}
\end{proposition}

\begin{proof}[Proof of \cref{explicitthick}]
	It follows from the definitions that $ \thi_{[-1,0]}(U) \subset \thi_b(U) \cap \Kproj.$ Let $X \in \left(\U[-1] * \U * \U[1]\right) \cap \Kproj$, then there exists a triangle
	\begin{equation}\label{primertriangulo}
		 V[-1] \rightarrow X \rightarrow Y \dashrightarrow V
	\end{equation}
	where $V \in \U$ and $Y \in \U * \U[1]$. In particular $Y$ is a $2$-term complex in $\U * \U[1]$. By definition, there exists a triangle 
	\begin{equation}\label{segundotriang}
		\bar{U}' \rightarrow Y \rightarrow \bar{U}[1] \dashrightarrow \bar{U}'[1]
	\end{equation}
	with $\bar{U}, \bar{U}' \in \U$. A rotation of the previous triangle gives a conflation $\bar{U} \rightarrowtail \bar{U}' \twoheadrightarrow Y $ in $\Kproj$, which implies that $Y \in \thi_{[-1,0]}(U)$. By rotating the triangle \ref{primertriangulo}, we get $X \rightarrowtail Y \twoheadrightarrow V$, implying that $X \in \thi_{[-1,0]}(U)$. Thus $$\left(\U[-1] * \U * \U[1]\right) \cap \Kproj \subset \thi_{[-1,0]}(U).$$
	
	We show now that $\thi_b(U) \cap \Kproj \subset \U[-1] * \U * \U[1]$. Since $\U$ is presilting, by \cref{thickdes} we have that 
	\[ \begin{split}
		\thi_b(U) & = \thi_b(\U)  = \\ 
		& = \bigcup_{n \geq 0} \U[-n] * \U[1-n]* \cdots * \U[n-1] * \U[n].
	\end{split} \]
	Let $X \in \thi_b(U) \cap \Kproj$. Then there exists $n$ such that $X \in \U[-n] * \U[1-n]* \cdots * \U[n-1] * \U[n] $. Since taking extensions is an associative operation, we can find a triangle 
	$$U'[-n] \xrightarrow{f} X \rightarrow X' \dashrightarrow U'[1-n]$$
	where $U' \in \U $ and $X' \in \U[1-n]* \cdots * \U[n-1] * \U[n]$. Suppose that $n>1$. Since both $X$ and $U$ are $2$-term complexes, $X \in  \prescript{\perp}{}{\U[\leq 2]}\cap \U[\geq 2]^\perp $. This implies that $f = 0$ and thus, $X$ is a direct summand of $X'$. By \cref{thickdes} i), we know that $\U[1-n]* \cdots * \U[n-1] * \U[n] = \add(\U[1-n]* \cdots * \U[n-1] * \U[n])$ and thus $X \in \U[1-n]* \cdots * \U[n-1] * \U[n]$. By applying the previous argument whenever $i-n < -1$, we can deduce that $X \in \U[-1] * \U * \cdots * \U[n-1] * \U[n]$. Using again the associativity of taking extension, we can find a triangle
	\[ X'' \rightarrow X \xrightarrow{g} U''[n] \dashrightarrow X''[1].\]
	where $X'' \in \U[-1] * \U * \cdots * \U[n-1]$ and $U'' \in \U$. Once more, if $n >1$ we deduce that $g=0$, since both $X$ and $U''$ are $2$-term complexes, and thus $X$ is a direct summand of $X'' \in \U[-1] * \U * \cdots * \U[n-1] = \add(\U[-1] * \U * \cdots * \U[n-1])$. We conclude that $X \in \U[-1] * \U * \cdots * \U[n-1]$. Applying this argument recursively whenever $n-i>1$, we finally get that $X \in \U[-1] * \U * \U[1]$. 
	We conclude that 
	\begin{align*}
		\left(\U[-1] * \U * \U[1]\right) &\cap \Kproj \subset  \thi_{[-1,0]}(U)  \\ &\subset \thi_b(U) \cap \Kproj  \subset \left(\U[-1] * \U * \U[1]\right) \cap \Kproj ,
	\end{align*}
	which gives the result. 
	
\end{proof}

\begin{corollary}\label{thickisthick}
	Let $\HH \subset \Kproj$ be a thick subcategory and consider $U \in \HH$ a 2-term presilting object. If  $\thi_b(U) = \thi_b(\HH)$, then $\thi_{[-1,0]}(U) = \HH$.
\end{corollary}
\begin{proof}
	
	Suppose $\thi_b(U) = \thi_b(\HH)$ and let $\U = \add(U)$. Since $U \in \HH$ and since $\HH$ is thick, we have that $\thi_{[-1,0]}(U )\subset \HH$. For the other inclusion note that $\HH \subset \thi_b(\HH) \cap \Kproj = \thi_b(U) \cap \Kproj$. By \cref{explicitthick}, $\thi_b(U) \cap \Kproj = \thi_{[-1,0]}(U)$, which gives the result. 
\end{proof}

\subsection{Silting reduction in $\Kproj$}

Before introducing the notion of reduction in $\Kproj$, let us first recall O.~Iyama and D.~Yang's additive description of the reduction of a triangulated category with respect to a presilting subcategory. 

\begin{theorem}\cite[Theorem 1.1]{iyama2018silting} \label{iyama_theo_reduction}
	Let $\D$ be a triangulated category and let $\U$ be a presilting subcategory of $\D$ satisfying certain mild assumptions\footnote{These assumptions are satisfied if, for instance, $\D$ is Hom-finite over a field and $\U= \add(U)$ for certain $U \in \D$ that can be completed into a silting object. Since we are working in the context where $\K$ is an algebraic $0$-Auslander reduced extriangulated category, and thus equivalent to $\Kproj$ for certain finite-dimensional algebra $\Lambda$ \cite{chen20230}, the needed assumptions hold (for more on these hypotheses, see \cite[Section 3.1]{iyama2018silting}). }. Let $\J_\U = \D/\thi_\D(\U)$ the triangle quotient of $\D$ with respect to $\U$. Let $\Z_\U = (^{\perp_\D} \U[>0]) \cap \U[<0]^{\perp_\D})$. Then the additive quotient $\Z_{\U}/[\U]$ has a natural structure of a triangulated category and we have a triangle equivalence $$\Z_\U/[\U] \xlongrightarrow{\bar{\rho}} \J_\U,$$ where $\bar{\rho}$ is induced by the functor $\Z_\U\subset \D \xrightarrow{\rho} \D/\thi_\D(\U) = \J_\U$.
\end{theorem}

\begin{theorem}\cite[Theorem 3.7]{iyama2018silting}\label{bij_silting_silting_iyama} Under the assumptions of \cref{iyama_theo_reduction}, the functor $\rho : \D \rightarrow\J_\U$ induces a bijection between the sets of presilting subcategories in $\J_\U$ and presilting subcategories in $\D$ containing $\U$. Moreover, a subcategory $\mathcal{P} \subset \D$ containing $\U$ is silting if and only if $\rho(\mathcal{P})$ is silting in $\J_\U$. 
\end{theorem}

From now on $\D = \Kb$ and $\U= \add(U)$ where $U$ is a $2$-term presilting object in $\Kb$. In this setting, the category $\J_\U$ of \cref{iyama_theo_reduction} has the following explicit description. 

\begin{proposition}\cite{neeman1992connection,borve2021two} Let $\Lambda$ be a finite-dimensional algebra and let $U$ be a basic $2$-term presilting complex in $\Kproj$ with Bongartz completion $T_U$. Then $\J_\U=  \Kb/\thi_b(\U)$ is equivalent to the category of perfect complexes $\per(C_U)$ where $C_U$ is the dg algebra $\EEn_{\D}(T_U)/{\langle e_{U}\rangle}$.
\end{proposition}
\begin{remark}\label{jasoo_red_algebra}
	The choice of $C_U$ as notation is intentional. Let $(M,P)$ be the support $\tau$-rigid pair associated to $U$, then by \cite[Theorem 4.12 b)]{jasso2015reduction} \[ H^0(C_U) = \En_{\J_\U}(T_U) \simeq \En_\Lambda(H^0(T_U))/{\langle e_{H^0(U)}\rangle}, \]
	where $\En_\Lambda(H^0(T_U))/{\langle e_{H^0(U)}\rangle} = C_{(M,P)}$  is the algebra associated to the $\tau$-tilting reduction of $\Mod \Lambda$ by $(M,P)$ as described by G.~Jasso in \cite{jasso2015reduction}.
\end{remark}

The following result is a weaker version of \cite[Lemma 3.4]{iyama2018silting} which will be essential four our results.  

\begin{proposition}\cite[Lemma 3.4]{iyama2018silting}\label{extension} The functor $\Kb \xrightarrow{\rho} \J_\U$ induces a bijective map 
	\[\Hom_\D(X, Y[i]) \rightarrow \Hom_{\J_\U}(X, Y[i])\]
	for every $i > 0$ and $X, Y \in \Z_\U$.
\end{proposition} 

We denote by $\langle 1 \rangle$ the shift functor in $\J_\U$ induced by $[1]$. We note that any triangle $X \xrightarrow{\bar{f}} Y \xrightarrow{\bar{g}} Z \dashrightarrow X\langle 1 \rangle$ in $\J_\U$ is the image of a triangle in $\Z_{\U}/[\U]$, which is in turn isomorphic to triangles obtained from a commutative diagram of the form
\begin{center}
	\begin{tikzcd}
	X \arrow[r, "f"] \arrow[d, equal] & Y \arrow[r, "g"] \arrow[d] & Z \arrow[r, dashed] \arrow[d] &X[1] \arrow[d, equal] \\
	X \arrow[r, "p_U"] & U_X \arrow[r] & X\langle 1 \rangle \arrow[r, dashed] &X[1] \, ,
	\end{tikzcd}
\end{center}
where $p_U$ is a minimal left $\U$-approximation of $X$. This fact says in particular that if $\C \subset \Z_{\U}$ is thick in $\Z_\U$, this remains true for $\C/[U]$ in $\Z_\U/[U]$ and thus for $\rho(\C) \subset \rho(\Z_\U)$.

\smallskip
The following lemma was originally shown by O.~Iyama and D.~Yang as a step towards proving \cref{iyama_theo_reduction}. In their context, they establish that for any $X \in \D$ there exists $Y \in \Z_\U$ such that $X \simeq Y$ in $\J_\U$. In the following lemma, we adapt their arguments to show that if $\D = \Kb$, and $X \in \Kproj$, then $Y$ can be chosen from $\Z_\U \cap \Kproj$.

\begin{lemma}\cite[Lemma 3.3]{iyama2018silting}
	Let $\Lambda$ be a finite-dimensional algebra and let $U$ be a basic $2$-term presilting object in $\Kproj$. For any $X \in \Kproj$ there exists $Y \in \Z_\U^{[-1,0]} = \Z_\U \cap \K^{[-1,0]}(\proj \Lambda)$ satisfying $X \simeq Y$ in $\J_\U$. 
\end{lemma}
\begin{proof}
	Let $\U= \add(U)$, with $U$ basic and presilting. First note that $$\Z_\U= \{ X \in \D \ | \ \Hom_\D(X, U[i]) = 0 = \Hom_\D(U[-i], X)  \ \forall i > 0 \}.$$ Since $\Hom_\D(Y, Y[i]) = \Hom_\D(Y[-i], Y) = 0 $ for all $i \geq 2$ and all $Y \in \Kproj$, then $$\Z_\U^{[-1,0]} = \{X \in \Kproj \ | \ \EE(X, U) = 0 = \EE(U, X)\}.$$ 
	
	Recall that $\K^{[-1,0]}(\proj \Lambda) = \Lambda * \Lambda[1]$ and that $\Lambda$ and $T_U$ are isomorphic in $\J_\U$ where $T_U$ is the Bongartz completion of $U$ into a silting complex in $\Kproj$, which by definition lies in $\Z_\U^{[-1,0]}$. We will show that we can find $H \in \Z_\U^{[-1,0]}$ such that $T_U[1] \simeq H$ in $\J_\U$. Consider the conflation $\Lambda \rightarrowtail U' \twoheadrightarrow V \dashrightarrow \Lambda[1]$ where the first morphism is a minimal right $\U$-approximation of $\Lambda$, then by definition $V \in \Kproj$ and $V \simeq \Lambda[1] \simeq T_U[1]$ in $\J_\U$. In fact, $V \simeq T_U^c$ inside $\J_\U$, where $T^c_U$ is the Bongartz co-completion of $U$, that is, the basic silting complex satisfying that $\add(T^c_U) = \add(V \oplus U)$. Since $T^c_U \in \Kproj$ and $\EE(T^c_U, U) = 0 = \EE(U, T^c_U)$, we have that $T^c_U \in \Z_\U^{[-1,0]}$ and that $T_U * T^c_U \subset \Z_\U^{[-1,0]}$. Since the functor $\bar{\rho}$ is a triangle equivalence, by \cref{extension} we have that
	$$ \Z_\U^{[-1,0]}/[\U] \supseteq (T_U * T_U^c)/ [\U]  \xrightarrow{\bar{\rho}} \rho(\Lambda)*_\J \rho(\Lambda)\langle 1 \rangle \simeq \rho(\Lambda * \Lambda[1]) = \subseteq \J_\U.$$
	In particular, for each $X \in \Kproj$ there exists $X' \in T_U * T^c_U$ such that $\rho(X) \simeq \rho(X')$. 
\end{proof}

\begin{remark}
	We note that the full extension-closed subcategory $\Z_\U^{[-1,0]}$ is precisely the subcategory considered by M. Gorsky, H. Nakaoka and Y. Palu in \cite[Definition 2.7]{gorsky2023hereditary} to define the reduction of an extriangulated category with respect to a presilting object.
\end{remark}

\begin{lemma}\label{rhothick} Let $U \in \Kproj$ be a $2$-term presilting complex and $\U$ its additive closure. Consider $\Z_\U$ and $\rho : \D \rightarrow \D/\thi_b(\U)$ as in \cref{iyama_theo_reduction}. Let $\HH \subset \Kproj$ be a thick subcategory such that $U \in \HH$. Then 
	$$\frac{\HH \cap \Z_\U}{[U]} \simeq \rho(\HH).$$
In particular, $\rho(\HH)$ is thick inside the extriangulated category $$ \rho(\Kproj) \simeq \pero(C_U).$$
\end{lemma}
\begin{proof}
	The following argument follows closely the those in \cite[Proposition 3.2, Lemma 3.3]{iyama2018silting}. We are going to show that for every $X \in \HH$ there exists and object $Z \in \HH \cap \Z_\U \subset \Z_\U^{[-1,0]}$ such that $\rho(X) \simeq \rho(Z)$. Let $X \in \HH$ and take $X\xrightarrow{f} U'[1]$ a minimal left $\U[1]$-approximation of $X$ and let $Y = \CCone(f)$. We then can construct a conflation 
	\[ U' \rightarrowtail Y \twoheadrightarrow X \overset{f}{\dashrightarrow}, \]
	which implies that $Y \in \Kproj$ and $\Hom(Y, U[n]) = 0$ for all $n \geq 2$. Moreover, applying the functor $\Hom_\D(-,U[1])$ to the previous triangle we obtain an exact sequence 
	\[\Hom_\D(U'[1],U[1]) \xlongrightarrow{\Hom_\D(-,f)} \Hom_\D(X,U[1]) \rightarrow \Hom_\D(Y, U[1]) \rightarrow \Hom_\D(U', U[1]). \]
	Given that $\Hom_\D(U', U[1]) = 0$, since $U$ is presilting and since $\Hom_\D(-,f)$ is surjective because $f$ is a left $\U[1]$-approximation, we deduce that $\Hom_\D(Y, U[1]) = 0$, 
	Now, consider $U''[-1] \xrightarrow{g} Y$ a minimal right $\U[-1]$-approximation and let $Z = \Cone(g)$, then there exists a conflation
	\[ Y \rightarrowtail Z \twoheadrightarrow U'' \overset{g[1]}{\dashrightarrow}. \]
	In particular, $Z \in \Kproj$ and $Z \in U[<-1]^{\perp}$. As before, by applying the functor $\Hom_\D(U[-1],-)$, and using that $g$ is a right $\U[-1]$-approximation and $U$ silting, we can deduce that $ Z \in U[-1]^\perp$. But $Z$ is an extension between two objects in $^\perp U[1]$, which implies that $Z \in ^\perp \U[>0] \cap \U[<0]^\perp = \Z_\U$. Moreover, under the assumption that $U \in \HH$, both previous conflations give that both $Y$ and $Z$ lie in $\HH$. We get that $$\rho(X) \simeq \rho(Y) \simeq \rho(Z),$$ with $Z \in \HH \cap \Z_\U$. 
	Since $\HH \cap \Z_\U$ is thick in $\Z_\U^{[-1,0]}$ by \cref{interthick}, we get that \[\rho(\HH) \simeq \bar{\rho}(\HH \cap \Z_\U)  \subset \bar{\rho}(\Z_\U^{[-1,0]}) \simeq \rho(\Kproj)\] is thick in $\rho(\Kproj) \simeq \pero(C_U)$. 
\end{proof}

\section{Thick subcategories and cotorsion pairs of $g$-finite algebras}\label{sec_thickg}

In this section we assume that $\Lambda$ is a $g$-finite algebra and study the maps between cotorsion pairs, thick subcategories and presilting objects in $\Kproj$ introduced in \cite{garcia2023thick}. 

\subsection{Bijection between cotorsion and torsion classes}
\begin{definition}\cite[Definition 1.7]{pauksztello2023cotorsion}\label{def_cotorsion}
	Let $\K$ be an extriangulated category with extension bifunctor $\EE$. We say that a pair of subcategories $(\X, \Y)$ is a \textit{cotorsion pair} if they are both full and additive and they satisfy 
	\begin{enumerate}
		\item $\Y = \X^{\perp_1} = \{Y \in \K \ | \ \EE(X,Y) = 0 \ \forall \ X\in \X\}$.
		\item  $\X = {}^{\perp_1}\Y = \{X \in \K \ | \ \EE(X,Y) = 0 \ \forall \ Y\in \Y\}$. 
	\end{enumerate}
We say that $(\X, \Y)$ is \textit{complete} \cite[Definition 4.1]{nakaoka2019extriangulated}, if additionally $$\K= \Cone(\Y,\X) = \CCone(\Y, \X).$$ When $\K = \Kproj$, we will denote by $\ctor \Lambda$ the set of all cotorsion pairs in $\Kproj$ and by $\cctor \Lambda$ the subset of cotorsion pairs of $\Kproj$ which are complete. 
\end{definition}

In \cite{pauksztello2023cotorsion}, D.~Pauksztello and A.~Zvonareva showed that the functor $H^0 : \Kproj \rightarrow \Mod \Lambda$ induces a map between the set of cotorsion pairs in $\Kproj$ and that of \textit{torsion classes} in $\Mod \Lambda$, which we denote by $\tor \Lambda$. 

\begin{theorem}\cite[Proposition 3.6 and Theorem 3.7]{pauksztello2023cotorsion}\label{PZoriginal}
	Let $\Lambda$ be a finite-dimensional $\Bbbk$-algebra. Then the functor $H^0 : \Kproj \rightarrow \Mod \Lambda$ induces a well-defined map
	\begin{align*}
		\psi : \ctor \Lambda &\rightarrow \tor \Lambda \\
		(\X, \Y) &\mapsto H^0(\Y).
	\end{align*}
	Moreover, $\psi$ induces a bijection between the sets of complete cotorsion pairs in $\Kproj$ and functorially finite torsion classes in $\Mod \Lambda$.
\end{theorem}

The following is an extension of \cref{PZoriginal}. 

\begin{theorem}\label{cotortorbij}
	Let $\Lambda$ be a finite-dimensional $\Bbbk$-algebra. Then the functor $H^0: \Kproj \rightarrow \Lambda$ induces a bijection 
	\begin{align*}
		\psi  : \ctor \Lambda &\rightarrow \tor \Lambda \\
		(\X, \Y) &\mapsto H^0(\Y) .
	\end{align*}
\end{theorem}
	\begin{proof}
	By \cref{PZoriginal}, the only thing left to proof is that $\psi$ is always a bijection whose inverse map is given by
	\[ \T \mapsto \left(\prescript{\perp_1}{}{\Y_\T}, \Y_\T\right) \]
	where $\Y_\T = \{X \in \Kproj \ | \ H^0(X)  \in \T \}$. It suffices to show that $(\prescript{\perp_1}{}{\Y_\T})^{\perp_1}= \Y_\T$. 
	Remark that for any $X \in \Kproj$, $X \simeq X_M \oplus Q[1]$, where $X_M$ is the minimal projective presentation of $M = H^0(X)$ and $Q \in \proj \Lambda$. By \cite[{Lemma 2.6}]{plamondon2013generic}, for any $Y \in \Kproj$ we have that $\EE(X_M,Y) \simeq  D\Hom_\Lambda(H^0(Y), \tau M)$. We deduce that for any $X, Y \in \Kproj$,
	\begin{align*}\EE(X, Y)  \simeq \EE(X_M,Y) &\oplus \EE(Q[1],Y) \simeq D\Hom_\Lambda(H^0(Y), \tau M) \oplus \Hom_{[-1,0]}(Q,Y) \\ &\simeq D\Hom_\Lambda(H^0(Y), \tau M) \oplus \Hom_\Lambda(Q, H^0(Y)).
	\end{align*}
	This implies that for any additive subcategory $\HH \subset \Kproj$, we have that 
	\begin{align*}
		\HH^{\perp_1} &= \{X \in \Kproj \ | \EE(\HH,X) = 0\} \\
		 & = \add \left(\left\{X_N \ | \ N \in \prescript{\perp}{}{\tau H^0(\HH)} \cap \bigl(\HH \cap \add(\Lambda[1])\bigr)[-1]^\perp \right\} \cup \add(\Lambda[1]) \right) \\
		 \prescript{\perp_1}{}{\HH} &= \{X \in \Kproj \ | \ \EE(X, \HH) = 0\}\\
	& = \add \left( \left\{ X_M \ | \ \tau M \in H^0(\HH)^\perp \right\} \cup \left(\add(\Lambda) \cap \prescript{\perp}{}{H^0(\HH)}\right)[1] \right).
	\end{align*}
	Now let $\T \subset \Mod \Lambda$ be a torsion class and let $\Y_\T = \{X \in \Kproj \ | \  H^0(X) \in \T \}$. Since $\T$ is closed under direct summands and for every $X, Y \in \Kproj$, $H^0( X\oplus Y) \simeq H^0(X)\oplus H^0(Y)$, we readily see that $\Y_\T$ is additive. This implies that 
	\[\prescript{\perp_1}{}{\Y_\T} = \add \left( \left\{ X_M \ | \ \tau M \in \T^\perp \right\} \cup \left(\add(\Lambda) \cap \prescript{\perp}{}{\T}\right)[1] \right). \]
	Hence
		\[H^0(\prescript{\perp_1}{}{\Y_\T}) = \{M \in \Mod \Lambda \ | \ \tau M \in \T^\perp\} = \add\left(\tau^{-1}(\T^\perp) \cup \add(\Lambda)\right)\]
	and
		\[ \prescript{\perp_1}{}{\Y_\T} \cap \add(\Lambda [1]) = \left(\add(\Lambda) \cap \prescript{\perp}{}{\T}\right)[1]. \]
	Finally, we get that 
	\begin{align*}\label{perpYperp}
		 (\prescript{\perp_1}{}{\Y_\T})^{\perp_1} &= \\
		 = &\add \left(\left\{X_N \ | \ N \in \prescript{\perp}{}{\tau H^0(\prescript{\perp_1}{}{\Y_\T})} \cap \left(\prescript{\perp_1}{}{\Y_\T} \cap \add(\Lambda[1])\right)[-1]^\perp \right\} \cup \add(\Lambda[1]) \right) \\
		 = &\add \left(\left\{ X_N \ | \ N \in \prescript{\perp}{}{\tau\left(\tau^{-1}(\T^\perp) \cup \add(\Lambda)\right)} \cap \left(\add(\Lambda) \cap \prescript{\perp}{}{\T}\right)^\perp \right\} \cup \add(\Lambda[1]) \right)
	\end{align*}
The only thing left to prove is that \[\prescript{\perp}{}{\tau\left(\tau^{-1}(\T^\perp) \cup \add(\Lambda)\right)} \cap \left(\add(\Lambda) \cap \prescript{\perp}{}{\T}\right)^\perp = \T. \] 
Note that $\tau\left(\tau^{-1}(\T^\perp) \cup \add(\Lambda)\right) = \tau\left(\tau^{-1}(\T^\perp)\right)$. Since  $\tau\left(\tau^{-1}(\T^\perp)\right)$ and $\T^\perp \setminus (\T^\perp \cap \inj \Lambda)$ have the same indecomposables, we get that $$\prescript{\perp}{}{\tau\left(\tau^{-1}(\T^\perp) \cup \add(\Lambda)\right)} = \prescript{\perp}{}{(\T^\perp \setminus (\T^\perp \cap \inj \Lambda))}.$$ Given that $\T$ is a torsion class and hence $\T = \prescript{\perp}{}{(\T^\perp)}$, we have that $$\T \subset \prescript{\perp}{}{\left(\T^\perp \setminus (\T^\perp \cap \inj \Lambda)\right)} \cap  \left(\add(\Lambda) \cap \prescript{\perp}{}{\T}\right)^\perp. $$

For the other inclusion, let $M$ in $\prescript{\perp}{}{\left(\T^\perp \setminus (\T^\perp \cap \inj \Lambda)\right)} \cap  \left(\add(\Lambda) \cap \prescript{\perp}{}{\T}\right)^\perp$. We want to prove that $M \in \prescript{\perp}{}{(\T^\perp)}$. Let $L$ be an indecomposable in $\T^\perp$. If $L$ is not injective, then $\Hom_\Lambda(M,L) = 0$ by definition of $M$. Assume thus that $L = I_i$, where $I_i \in \T^\perp \cap \inj \Lambda$ is the injective envelope of the simple $S_i$. Recall that for any $N \in \Mod \Lambda$, $\Hom(N, I_i) = 0$ if and only if $\Hom_\Lambda(P_i, N) = 0$, where $P_i$ is the projective cover of $S_i$. Thus $I_i \in \T^\perp \cap \inj \Lambda$ if and only if $\Hom_\Lambda(\T, I_i) = 0 = \Hom_\Lambda(P_i, \T)$, which is equivalent to $P_i \in \add(\Lambda) \cap \prescript{\perp}{}{\T}$. But $M \in \left( \add(\Lambda) \cap \prescript{\perp}{}{\T}\right)^\perp$ as well, so we have that
\[\Hom_\Lambda(P_i, M) = 0 = \Hom_\Lambda(M, I_i).\]
We conclude that $M \in \prescript{\perp}{}{(\T^\perp)} = \T$. In particular we get that 
\[ (\prescript{\perp_1}{}{\Y_\T})^{\perp_1} =\add \left( \{ X_N \ | \ N \in \T \ \} \cup \add(\Lambda[1])\right) = \Y_\T\]
which implies that $(\prescript{\perp_1}{}{\Y_\T}, \Y_\T)$ is a cotorsion pair. 
\end{proof}

\cref{cotortorbij} induces a ``mirror" of \cref{demonet_finite} in the category of projective presentations: 
\begin{corollary}
	Let $\Lambda$ be a finite-dimensional $\Bbbk$-algebra. The following are equivalent:
	\begin{enumerate}
		\item $\Lambda$ is $g$-finite.
		\item There exist finitely many complete cotorsion pairs in $\Kproj$.
		\item All cotorsion pairs in $\Kproj$ are complete.
	\end{enumerate}
\end{corollary}
\begin{proof}
	The implications follow from \cref{cotortorbij}, \cref{PZoriginal} and \cref{demonet_finite}.
\end{proof}
\begin{example}[Kronecker Quiver] Consider $\Lambda$ to be the path algebra the Kronecker quiver \[\begin{tikzcd} 1 \arrow[r,shift left=.75ex,"\alpha"] \arrow[r,shift right=.75ex,"\beta", swap] & 2 \end{tikzcd}. \]
Recall that all torsion classes of $\Lambda$ can be described as the additive closure of one of the following four types of subsets of indecomposable modules :
\begin{enumerate}
	\item Any final part of the preinjective component of the Auslander-Reiten quiver of $\Mod \Lambda$. 
	\item All preinjectives and a subset of the tubes of the Auslander-Reiten quiver of $\Mod \Lambda$.
	\item All preinjectives, all tubes and a final part of the postprojective component of the Auslander-Reiten quiver of $\Mod \Lambda$
	\item The module $P_1$ whose dimension vector is $(1,0)$.
\end{enumerate}

\begin{figure}[h!]
	\makebox[\textwidth][c]{\includegraphics[width=1.3\textwidth]{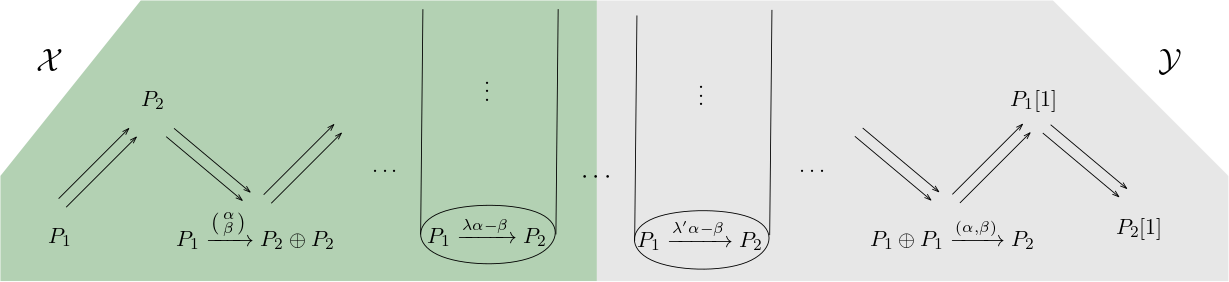}}
	\caption{\centering Example of a non-complete cotorsion pair $(\X, \Y)$ in $\Kproj$ when $\Lambda$ is the Kronecker algebra. Here $\lambda$ and $\lambda'$ lie in two disjoint parts of $\mathbb{P}(\Bbbk)$. The object at the base of the tube corresponding to the point at infinity is $P_1 \xrightarrow{\alpha} P_2$.}
	\label{kro}
\end{figure}

From these torsion classes, all are functorially finite with the exception of those of type (2). As a result of \cref{cotortorbij}, we can explicitly describe all cotorsion pairs of $\Kproj$. All complete ones are easily described by the results of T.~Adachi and M.~Tsukamoto in \cite{adachi2022hereditary}, who showed that any complete cotorsion pair $(\X, \Y)$ in $\Kproj$ satisfies that $\X= \CCone(\add(U), \add(U))$ and $\Y= \Cone(\add(U), \add(U))$ for a certain two-term silting object $U \in \Kproj$ satisfying $\X \cap \Y= \add(U)$. A non-complete cotorsion pair $(\X, \Y)$ is presented in \cref{kro}. Note that in this case $\X \cap \Y$ is always trivial.

\end{example}

\subsection{All thick subcategories have enough injectives and projectives}
Let $\Lambda$ be a finite-dimensional $\Bbbk$-algebra. A \textit{wide subcategory} of $\Mod \Lambda$ is a full additive subcategory of $\Mod \Lambda$ that is closed under extensions, kernels and cokernels. We say that a wide subcategory is \textit{left finite}, if the smallest torsion class containing it is functorially finite. Left finite wide subcategories were defined by F.~Marks and J.~{\v{S}}{\v{t}}ov{\'\i}{\v{c}}ek in \cite{marks2017torsion}, where it was shown that they are in bijection with functorially finite torsion classes. Given that the definition of a left finite wide subcategory is given by a condition on its associated torsion class, the following results follows immediately from \cref{demonet_finite}. 
\begin{corollary}\cite{demonet2019tilting, marks2017torsion}\label{allwideareleft}
	If $\Lambda$ is a g-finite finite-dimensional $\Bbbk$-algebra, then all wide subcategories are left finite.
\end{corollary}

In \cite{garcia2023thick}, new maps between cotorsion pairs, thick subcategories and silting complexes in $\Kproj$ were introduced. These maps mirror the maps between torsion classes, wide subcategories and support $\tau$-tilting pairs of $\Mod \Lambda$ introduced in \cite{adachi2014tilting}, \cite{ingalls2009noncrossing} and \cite{marks2017torsion}; and restrict to bijections for certain subsets of these sets of objects. We denote by $\thi \Lambda$ the set of thick subcategories of $\Kproj$ and by $\fthi \Lambda$ the set of those thick subcategories that have enough injectives with respect to the extriangulated structure inherited of $\Kproj$.

\begin{theorem}\cite[Theorem 3.1]{garcia2023thick}\label{beta}
	Let $\Lambda$ be a finite-dimensional $\Bbbk$-algebra. There exists a well defined map
	$$\ctor \Lambda  \overset{\beta}{\longrightarrow} \thi \Lambda $$
 given by $$\beta(\X, \Y) = \{ X \in \X \ | \ \forall \ \text{conflation } X \rightarrowtail X' \twoheadrightarrow X'' \text{ such that } X' \in \X, \text{ then } X'' \in \X\}.$$ 
for any $(\X, \Y) \in \ctor \Lambda$. Moreover, $\beta$ induces a bijection between the set of complete cotorsion pairs and that of thick subcategories with enough injectives. 
\end{theorem}

\begin{theorem}\cite[Theorem 3.15]{garcia2023thick}\label{theo-thick-wide}
	Let $\Lambda$ be a finite-dimensional $\Bbbk$-algebra. There exists a well defined map
	\[\thi \Lambda \overset{\Ww}{\longrightarrow} \wide \Lambda\]
	such that, when restricted to thick subcategories with enough injectives and left finite wide subcategories, it fits in the following commutative diagram

	\vspace{1mm}
	\begin{center}
		\begin{tikzcd}[cramped,column sep=large, row sep=small]
			& & \tsilt \Lambda \arrow[dr, "\thi(U_\rho)"] & & \\
			&  \cctor\Lambda \arrow[ur, "\Psi"]  \arrow[rr, "\beta", crossing over]& & \fthi \Lambda& \\ 
			\mathclap{\K^{[-1,0]}(\proj \Lambda)} & & & &  \\
			{} \arrow[rrrr, dashed, no head, crossing over]& & & & {} \\
			& & & & \mathclap{\Mod \Lambda} \\
			& \ftor\Lambda \arrow[from= uuuu, "\Phi", crossing over, crossing over clearance=7ex] & & \fwide\Lambda \arrow[from= ll, "\alpha"] \arrow[from= uuuu, crossing over, crossing over clearance=7ex,"\Ww"] & 
		\end{tikzcd}
	\end{center} 
	In particular, $\Ww$ and the map taking any $U \in \tsilt \Lambda $ to the thick category $\thi(U_\rho) \in \fthi \Lambda$ are bijective. Here, $U_\rho$ is the basic direct summand of $U$ satisfying $\add(U_\rho) = \add(U')$, where $U'$ is such that $U' \rightarrow \Lambda[1]$ is a minimal right $\add(U)$-approximation of $\Lambda[1]$. 
\end{theorem}

\begin{corollary}\cite[Corollary 1.1]{garcia2023thick}\label{main_coro}
	There are explicit bijections between:
	\begin{enumerate}[(i)]
		\item Isomorphism classes of basic silting objects in $\K^{[-1,0]}(\proj \Lambda)$.
		\item Complete cotorsion pairs in $\K^{[-1,0]}(\proj \Lambda)$.
		\item Thick subcategories in $\K^{[-1,0]}(\proj \Lambda)$ with enough injectives.
	\end{enumerate}	These bijections fit in the following commutative diagram: 
\end{corollary}

\begin{figure}[h!]
	\centering
	\begin{tikzcd}[cramped,column sep=large, row sep=small]
		& & \tsilt \Lambda \arrow[ddddd, "{H^0}" {yshift= 2ex}, "\text{\cite{adachi2014tilting}}"' {yshift= 1.5ex}] \arrow[dr, "\thi(U_\rho)"] & & \\
		&  \ctor \Lambda \arrow[from = ur, "\textbf{\cite{adachi2022hereditary} }"' {xshift= 2ex}]  \arrow[rr, "\beta" {xshift= 3ex}, crossing over]& & \fthi \Lambda & \\ 
		\mathclap{\Kproj} & & & &  \\
		{} \arrow[rrrr, dashed, no head, crossing over]& & & & {} \\
		& & & & \mathclap{\Mod \Lambda} \\
		& & \tautilt \Lambda \arrow[dl, " \text{\cite{adachi2014tilting} }"' {xshift=2ex}] \arrow[dr]& & \\
		& \ftor \Lambda \arrow[from= uuuuu, "{H^0}", "\text{\cite{pauksztello2023cotorsion}}"', crossing over, crossing over clearance=7ex] & & \fwide \Lambda \arrow[from= ll, "\alpha", "\text{\cite{marks2017torsion}}"'] \arrow[from= uuuuu, crossing over, crossing over clearance=7ex,"\Ww"] & 
	\end{tikzcd}
	\label{prisma}
\end{figure}

Unlike being a left finite wide subcategory, having enough injectives is a characteristic of thick subcategories that is inherent to them. In this section, we establish an equivalence between being $g$-finite and having finitely many thick subcategories. Furthermore, we will show that if $\Lambda$ is $g$-finite, then every thick subcategory in $\Kproj$ is generated by a $2$-term presilting complex, and that we can choose it to be injective in the thick subcategory it generates. 

First, we will need to show that if $\Lambda$ is a $g$-finite algebra, then all non-trivial thick subcategories in $\Kproj$ contain a non-zero presilting complex. For any $X \in \Kproj$, we denote by $[X]$ the class of $X$ in the Grothendieck group $K_0(\Kproj)$.

\begin{proposition}\label{presilt_in_thick} Let $\Lambda$ be a finite-dimensional $\Bbbk$-algebra. Suppose that $\Lambda$ is $g$-finite and let $\HH \subset \Kproj$ be a full additive subcategory that is closed under extensions. If $X \in \HH$, then $\HH$ contains a presilting object $0 \neq U$ such that $[X] = [U]$ or $\HH$ contains a non-zero $P \in \proj \Lambda$ and its shift $P[1]$.
\end{proposition}

To establish \ref{presilt_in_thick}, we will employ an algebraic-geometric result concerning the varieties of bounded complexes of projective modules over a finite-dimensional algebra. Let $p<q  \in \mathbb{Z}$ and consider $\C^{[p,q]}(\proj \Lambda) \subset \C^b(\proj \Lambda)$ the category of complexes of projective $\Lambda$-modules concentrated in degrees in the interval $[p,q]$. Fix a set of representatives $\{P_i\}_{1\leq 1\leq n}$ of the isoclasses of indecomposable projective $\Lambda$-modules, then for any choice of $\bar{l} = (l_p, \cdots, l_{p+j}, \cdots, l_{q}) \in (\mathbb{Z}_{\geq 0}^n)^{q-p+1}$ where $l_j = (l_{j,1}, l_{j,2}, \cdots, l_{j,n})$, we define $R_{\bar{l}}$ to be the closed subvariety 
\[R_{\bar{l}} \subset \prod_{j = 0}^{q-p-1} \Hom_\Lambda\left(\bigoplus_{i=1}^n P_i^{\oplus l_{p+j , i}} , \bigoplus_{i=1}^n P_i^{\oplus l_{p+j+1 , i}} \right)  \]
defined by the relation $f_{p+i+1} \circ f_{p+i} = 0$ for all $0 \leq i \leq q-p-2$. In other words, $R_{\bar{l}}$ parametrizes all complexes in $\C^{[p,q]}(\proj \Lambda)$ with $\bigoplus_{i=1}^n P_i^{\oplus l_{p+j , i}}$ in position $p + j$. The variety $R_{\bar{l}}$ is equipped with a group action of \[G_{\bar{l}} = \prod_{j = 0}^{q-p} \Aut_\Lambda\left(\bigoplus_{i=1}^n P_i^{\oplus l_{p+j , i}} \right)\] given by 
\[ (g_{p+j})_{0 \leq j \leq q-p} \cdot (f_{p+i})_{0 \leq i \leq q-p-1} = \left( g_{p+i+1}f_{p+i}g_{p+i}^{-1}\right)_{0 \leq i \leq q-p-1} .\] 
 
\begin{theorem}\cite[Theorem 2]{bernt2005degenerations} \label{degeneration}
	Let $\Lambda$ be a finite-dimensional $\Bbbk$-algebra and suppose that that  $\Bbbk = \bar{\Bbbk}$. Let  $\bar{l} = (l_p, \cdots, l_{p+j}, \cdots, l_{q}) \in (\mathbb{Z}_{\geq 0}^n)^{q-p+1}$ for some $p < q \in \mathbb{Z}$. Let $N, M \in R_{\bar{l}}$ such that $N \in \overline{G_{\bar{l}} \cdot M}$. Then there exists $m \in \mathbb{N}$ and exact sequences in $\C^b(\proj \Lambda)$
	\[ 0 \rightarrow N_0 \rightarrow N_{i+1} \rightarrow N_i \rightarrow 0\]
	for any $ 0 \leq i \leq m-1$ such that $N_0 = N$ and $N_{m} \simeq M \oplus N'$ for some $N' \in \C^{[p,q]}(\proj \Lambda)$. 
\end{theorem}

We will also make use of the following known result, which follows from \cite{demonet2019tilting}.

\begin{proposition}\cite{demonet2019tilting}\label{completefan} Suppose that $\Lambda$ is $g$-finite. Then for any $\theta \in K_0(\Kproj)$ there exists a presilting object $X \in \Kproj$ such that $[X] = \theta$. 
\end{proposition}

\begin{theorem}\cite[Theorem 6.5]{demonet2019tilting}\label{samegvector} Let $\Lambda$ be a finite-dimensional algebra and let $U$ and $V$ be $2$-term presilting complexes in $\Kproj$. Then $[U] = [V]$ if and only if $U \simeq V$. 
\end{theorem}

\begin{proof}[\textit{Proof of \cref{presilt_in_thick}}]
Note that if $\HH = \{0\}$ then the propositions follows immediately. Suppose then that there is $ 0 \neq X \in \HH$. 

\fbox{\textit{Case $[X] = 0$: }}  If $[X] = 0$, given that the only $2$-term presilting complex with zero $g$-vector is $0$, we have to show that there exists a non-zero projective module $P$ such that $\begin{tikzcd}[cramped, sep=small] P \arrow[d] \\ 0 \end{tikzcd} \in \HH$. Since $[X] = 0$, there exists $P \in \proj \Lambda$ such that $ X \simeq \begin{tikzcd}[cramped, sep=small] P \arrow[d, "f"] \\ P \end{tikzcd}$. If $f$ is not radical, then $X$ is isomorphic to some $\complex{P'\oplus P''}{{\begin{psmallmatrix} f' & 0 \\ 0 & a \end{psmallmatrix}}}{P' \oplus P}$ with $a$ invertible. Since $\HH$ is additive we can assume that $\complex{P'}{f'}{P'} \in \HH$. Thus we can suppose that $f$ is radical and hence there exists $m \geq 1$ such that $f^m = 0$. Consider the morphism $\delta \in \Hom(X, X[1])$ given by the commutative diagram 
\begin{center}
	\begin{tikzcd}
	 0 \arrow[r] \arrow[d] & P \arrow[d, "f"] \\
	 P \arrow[r, "\Id_P"] \arrow[d, "f"] & P \arrow[d] \\
	 P \arrow[r] & 0
\end{tikzcd}
\end{center}

Since $f$ is radical, then $\delta \neq 0$ and its mapping cone $\complex{P \oplus P}{{\begin{psmallmatrix} -f & 0 \\ 1_P &f \end{psmallmatrix}}}{P \oplus P}$ is isomorphic to the complex in the rightmost column of the following diagram
\begin{center}
	\begin{tikzcd}[ampersand replacement=\&, row sep = large, column sep= large]
		P \oplus P \arrow[d, "{\begin{pmatrix} -f & 0 \\ \Id_P & f \end{pmatrix}}"'] \arrow[r, "{\begin{pmatrix} \Id_P & f \\ 0 & \Id_P \end{pmatrix}}", "\simeq"']\& P \oplus P \arrow[d, "{\begin{pmatrix} 0 & f^2 \\ \Id_P & 0 \end{pmatrix}}"'] \\ P \oplus P \arrow[r, "{\begin{pmatrix} \Id_P & f \\ 0 & \Id_P \end{pmatrix}}"', "\simeq"] \& P \oplus P .
	\end{tikzcd}
\end{center}
This implies that $X^{(1)} = \Cone(\delta)[-1] \simeq \begin{tikzcd}[cramped, sep=small] P \arrow[d, "f^2"] \\ P \end{tikzcd}$ belongs to the thick subcategory $\HH$ since it is an self-extension of $X$. By repeating this argument we can construct objects $X^{(i)} \simeq \begin{tikzcd}[cramped, sep=small] P \arrow[d, "f^{2i}"] \\ P \end{tikzcd}$ for every $i \geq 1$. By choosing $i$ such that $2i \geq m$ we conclude that $P \oplus P[1] \simeq \begin{tikzcd}[cramped, sep=small] P \arrow[d, "0 = f^{2i}"] \\ P \end{tikzcd}$ is in $\HH$, and since $\HH$ is closed under direct summands, then both $P$ and $P[1]$ are in $\HH$. 

\fbox{\textit{Case $[X] \neq 0$: }}
First suppose that $\Bbbk= \bar{\Bbbk}$. Let $X \in \HH$ such that $[X]\neq 0$, then $X = \begin{tikzcd}[cramped, sep=small] X^{-1} \arrow[d, "f"] \\ X^0 \end{tikzcd}$, where $[X^{-1}] \neq [X^0]$. Recall that we can decompose the $g$-vector of $X$ as $[X] = \theta_X^+ - \theta_X^-$, where $\theta_X^+ = \left(\max\{0, \theta_{X,i}\}\right)_{1 \leq i \leq n}$ and $\theta_X^- = \left(\max\{0, -\theta_{X,i}\}\right)_{1 \leq i \leq n}$. By \cref{completefan} there exist a $2$-term presilting complex $U$ whose $g$-vector is $[X]$. Take $u$ the point in $\Hom(P^{\theta_X^-}, P^{\theta_X^+})$ corresponding to $U$, which has an open dense orbit (see for instance \cite[Lemma 2.16]{plamondon2013generic}). Choose $Q \in \proj \Lambda$ such that $P^{\theta_X^-}  \oplus Q = X^{-1}$ and $P^{\theta_X^+}  \oplus Q = X^{0}$, then $u \oplus \Id_Q$ still has an open dense orbit in $\Hom(X^{-1}, X^0)$. Hence $\overline{G \cdot u \oplus \Id_Q} = \Hom(X^{-1}, X^0)$, and $X \in \overline{G \cdot U \oplus \begin{tikzcd}[cramped, sep=small] Q \arrow[d, equal] \\ Q \end{tikzcd}}$. From \cref{degeneration} we deduce that $U \oplus \begin{tikzcd}[cramped, sep=small] Q \arrow[d, equal] \\ Q \end{tikzcd}$ can be constructed as a direct summand of a sequence of self-extensions of $X$ in $\C^{[-1,0]}(\proj \Lambda)$. In particular, $U$ is contained in $\HH$ and $[U] = [X]$. This proves the result over an algebraically closed field.

\smallskip
Now let $\Bbbk$ be any field and let $\mathbb{K} = \bar{\Bbbk}$. We have a fully faithful functor $\C^b(\proj \Lambda)\otimes_\Bbbk \mathbb{K} \hookrightarrow \C^b(\proj \Lambda \otimes_\Bbbk \mathbb{K})$ induced by $- \otimes_\Bbbk \mathbb{K}$. Let $X \in \HH$ be as before and consider $\bar{X} = X \otimes_\Bbbk \mathbb{K} \in \K^{[-1,0]}(\proj \Lambda \otimes_\Bbbk \mathbb{K})$. By the previous argument we can find a $2$-term presilting complex $\bar{U} \in \K^{[-1,0]}(\proj \Lambda \otimes_\Bbbk \mathbb{K})$ that is a direct summand of an object obtained by a sequence of self-extensions of $\bar{X}$. Explicitly, there are exact sequences in $\C^{[-1,0]}(\proj \Lambda \otimes_\Bbbk \mathbb{K})$
\begin{equation}\label{cdprojA}
	\begin{tikzcd}
		0 \arrow[r] & \bar{X}^{-1}_0 \arrow[r] \arrow[d, "\bar{f}_0"] & \bar{X}^{-1}_{i+1} \arrow[r] \arrow[d, "\bar{f}_{i+1}"] & \bar{X}^{-1}_i \arrow[d, "\bar{f}_i"] \arrow[r] & 0\\
		0 \arrow[r] & \bar{X}^{0}_0 \arrow[r] & \bar{X}^{0}_{i+1} \arrow[r] & \bar{X}^{0}_i \arrow[r] & 0
	\end{tikzcd}
\end{equation}
for $0 \leq i \leq m-1$ where $\bar{X} = \begin{tikzcd}[cramped, sep=small] X^{-1}\otimes_\Bbbk \mathbb{K} \arrow[d, "f\otimes_\Bbbk \mathbb{K}"] \\ X^0 \otimes_\Bbbk \mathbb{K} \end{tikzcd} = \begin{tikzcd}[cramped, sep=small] \bar{X}^{-1}_0 \arrow[d, "\bar{f}_0"] \\ \bar{X}^0_0 \end{tikzcd}$ and such that $\bar{U}$ is a direct summand of $\bar{X}_m = \begin{tikzcd}[cramped, sep=small] \bar{X}^{-1}_m \arrow[d, "\bar{f}_m"] \\ \bar{X}^0_m \end{tikzcd}$. 
In particular, for $i = 0$ we get 

\begin{equation}\label{cdproj}
	\begin{tikzcd}
		0 \arrow[r] & \bar{X}^{-1}_0 \arrow[r] \arrow[d, "\bar{f}_0"] & \bar{X}^{-1}_{1} \arrow[r] \arrow[d, "\bar{f}_{1}"] & \bar{X}^{-1}_0 \arrow[d, "\bar{f}_0"] \arrow[r] & 0 \ \text{\textcolor{white}{.}}\\
		0 \arrow[r] & \bar{X}^{0}_0 \arrow[r] & \bar{X}^{0}_{1} \arrow[r] & \bar{X}^{0}_0 \arrow[r] & 0 \ .
	\end{tikzcd}
\end{equation}
By definition, both lines in the commutative diagram \ref{cdproj} are short exact sequences. Since $\bar{X}^{-1}_0 = X^{-1}_0 \otimes_\Bbbk \mathbb{K}$ and $\bar{X}^{0}_0 = X^{0}_0 \otimes_\Bbbk \mathbb{K}$ are projective modules in $\proj \Lambda \otimes_\Bbbk \mathbb{K}$, both exact sequences split. In particular 
\begin{gather*}
	\bar{X}^{-1}_1 \simeq  \left(X^{-1}_0 \otimes_\Bbbk \mathbb{K}\right) \oplus \left(X^{-1}_0 \otimes_\Bbbk \mathbb{K}\right)   \simeq \left(X^{-1}_0 \oplus X^{-1}_0\right) \otimes_\Bbbk \mathbb{K} \\
	\bar{X}^{0}_1 \simeq  \left(X^{0}_0 \otimes_\Bbbk \mathbb{K}\right) \oplus \left(X^{0}_0 \otimes_\Bbbk \mathbb{K}\right)   \simeq \left(X^0_0 \oplus X_0^0\right) \otimes_\Bbbk \mathbb{K} 
\end{gather*}
and thus 
\[\begin{split}
	\bar{f}_1 \in & \Hom_{\Lambda \otimes_\Bbbk \mathbb{K}}\left(\left(X^{-1}_0 \oplus X^{-1}_0\right) \otimes_\Bbbk \mathbb{K}, \left(X^{0}_0 \oplus X_{0}^0\right) \otimes_\Bbbk \mathbb{K}\right) \\ \simeq & \Hom_\Lambda\left(X^{-1}_0 \oplus X^{-1}_0, X^{0}_0 \oplus X_{0}^0\right)\otimes_\Bbbk \mathbb{K}.
\end{split}\]
We deduce that the exact sequence of complexes \ref{cdproj} is induced by an exact sequence of complexes in $\C^{[-1,0]}(\proj \Lambda)$ under the functor $- \otimes_\Bbbk \mathbb{K}$. By applying the same argument for every $0 \leq i \leq m-1$ and its respective exact sequence \ref{cdprojA}, we deduce that all self-extensions of $\bar{X} = X \otimes_\Bbbk \mathbb{K}$ lie in $\C^b(\proj \Lambda)\otimes_\Bbbk \mathbb{K}$. In particular, there exists $U \in \Kproj$ which is a direct summand of a sequence of self-extensions of $X$ such that $\bar{U} \simeq U \otimes_\Bbbk \mathbb{K}$. Moreover, $U$ is presilting if and only if $\bar{U}$ is (see for instance \cite[Proposition 6.6 b)]{demonet2019tilting}), and $[X] = [\bar{X}] = [\bar{U}] = [U]$, which finishes the proof.
\end{proof}

\begin{theorem}\label{gfinite-thick}
	Let $\Lambda$ be a $g$-finite, finite-dimensional $\Bbbk$-algebra. Let $\HH$ be any thick subcategory of $\Kproj$, then there exists a presilting complex $U \in \Kproj$ such that $\HH = thick_{[-1,0]}(U)$.
\end{theorem}
\begin{proof}
	Let $\Lambda$ be any $g$-finite finite-dimensional $\Bbbk$-algebra and let $\HH$ be a thick subcategory of $\Kproj$. If $\HH = \{0\}$ then we are done. If not, take $0 \neq X \in \HH$. By \cref{presilt_in_thick} there exists a presilting $0 \neq U$ in $\HH$ and we let $\thi_b(U)$ be the thick subcategory of $\Kb$ generated by $U$. By \cref{explicitthick} we know that $\thi_{[-1,0]}(U) = \thi_b(U) \cap \Kproj$. Consider $\J_U = \Kb /\thi_b(U) $ and $\rho : \Kb \rightarrow \J_U$ as in \cref{iyama_theo_reduction}. By \cref{rhothick}, we know that $\HH' = \rho(\HH)$ is a thick subcategory of $ \rho(\Kproj) \simeq \pero(C_U)$. 
	
 	Now consider the functor $$p_*: \pero(C_U) \rightarrow \K^{[-1,0]}(\proj H^0(C_U))$$ induced by the canonical projection $p : C_U \rightarrow H^0(C_U)$ as in \cref{indc_funct}. This functor induces an equivalence $p_*: \pero(C_U)/\mathcal{I} \rightarrow \K^{[-1,0]}(\proj H^0(C_U))$ where $\mathcal{I}$ is the ideal of morphisms that factor trough a morphism $X[1] \rightarrow Y$ for $X, Y \in \add(C_U)$. By \cref{thekey}, $p_*(\HH')$ is closed under extensions and direct summands. Moreover, if $\HH' \neq \{0\}$ then $p_*(\HH') \neq \{0\}$ since $p_*$ preserves isomorphism classes. 
	As recalled in \cref{jasoo_red_algebra}, $\Mod(H^0(C_U))$ is equivalent to the $\tau$-tilting reduction of $\Mod \Lambda$ associated to $U$, in particular $H^0(C_U)$ is $g$-finite by \cite[Theorem 3.16]{jasso2015reduction}. By applying \ref{presilt_in_thick} to the category $p_*(\HH') \subset \K^{[-1, 0]}(\proj H^0(C_U))$, we can find a $2$-term presilting object $0 \neq V' \in p_*(\HH')$, and thus there exists $V \in \HH'$ such that $V'=p_*(V)$ which is itself presilting by \cref{ind_silting}. Moreover, $V \in \rho(\HH) \simeq \bar{\rho}(\HH \cap \Z_\U)$, thus by \cref{bij_silting_silting_iyama} we know that $V= W \oplus U \in \HH\cap \Z_\U \subset \HH$, where $W \neq 0$ since we supposed $V'$ and thus $V$ to be non-zero in $\J_U$. By substituting $U$ by $V$ in the previous argument, we can find a sequence of $2$-term presilting complexes $(V_i)_{i \in \mathbb{N}}$ such that $\add(V_i) \subsetneq \add(V_{i+1})$. Since the number of indecomposables of a presilting $2$-term complex is bounded by $|\Lambda|$, the sequence must stabilize, that is for $i \gg 0$ $$\add(V_i) = \add(V_{i+1}) = \add(\bar{U}),$$ where $\bar{U} \in \Kproj$ is a basic presilting complex. Then $\HH = \thi_{[-1,0]}(\bar{U})$, which gives the result. 
\end{proof}

\begin{corollary}\label{finitethick}
	Let $\Lambda$ be a finite-dimensional algebra, then $\Lambda$ is $g$-finite if and only if there exist finitely many thick subcategories in $\Kproj$.
\end{corollary}
\begin{proof} 
	Suppose $\Lambda$ is $g$-finite. By \cref{gfinite-thick}, any thick subcategory is generated by a presilting complex $U \in \Kproj$, and since there are finitely many isomorphism classes of such $U$, then there exist finitely many thick subcategories. Conversely, if there are finitely many thick subcategories, then there are also finitely many of them with enough injectives. Given that \cref{main_coro} establishes a bijection between isomorphism classes $2$-term silting complexes and thick subcategories with enough injectives, we conclude that $\Lambda$ is $g$-finite.
	\smallskip
\end{proof}

\begin{example}
	Consider $\Lambda$ to be the path algebra the Kronecker quiver \[Q = \begin{tikzcd} 1 \arrow[r,shift left=.75ex,"\alpha"] \arrow[r,shift right=.75ex,"\beta", swap] & 2 \end{tikzcd}. \]
	\cref{gfinite-thick} provides yet another illustration of the fact that $\Lambda$ is $g$-infinite. Indeed, the tubes in $\Kproj$ provide a family of thick subcategories which do not contain a non-zero presilting object and that have no non-zero object whose $g$-vector coincides with that of a presilting object of $\Kproj$.
\end{example}

Before moving on to prove \cref{allthick}, we first recall the definition of the map $\Ww$ introduced in \cref{presilt_in_thick}. 
\begin{definition}
	Let $\C \subset \Mod \Lambda$ and $\HH \subset \Kproj$ be subcategories of $\Mod \Lambda$ and $\Kproj$ respectively. We define $\Tt(\C)$ to be the full subcategory of $\Kproj$ whose objects are all complexes $X = \complex{X^{-1}}{x}{X^0}$ such that the $\Bbbk$-linear map $$\Hom_\Lambda(x, M) : \Hom_\Lambda(X^0, M) \rightarrow \Hom_\Lambda(X^{-1}, M)$$ is an isomorphism. Similarly, we define $\Ww(\HH)$ as the full subcategory of modules $M$ such that $\Hom_\Lambda(x, M)$ is an isomorphism for all complexes $X = \complex{X^{-1}}{x}{X^0} \in \HH$. 
\end{definition}

\begin{proposition}\label{W_T_perp}
	Let $\C \subset \Mod \Lambda$ and $\HH \subset \Kproj$ be subcategories, then
	 \begin{gather*}
	 	\Ww(\HH) = \HH^{\perp_\mathbb{Z}} \cap \Mod \Lambda \\
	 	\Tt(\C)  = \prescript{\perp_\mathbb{Z}}{}{\C} \cap \Kproj
	 \end{gather*}
where \begin{gather*}
	\HH^{\perp_\mathbb{Z}} = \{Y \in \D \ | \ \Hom_\D(\HH, Y[i]) = 0 \ \forall \ i \in \mathbb{Z}\} \\
	\prescript{\perp_\mathbb{Z}}{}{\C} = \{Y \in \D \ | \ \Hom_\D(Y[i], \C) = 0 \ \forall \ i \in \mathbb{Z}\}
\end{gather*} and $\D = \D^b(\Mod \Lambda)$. 
\end{proposition}
\begin{proof}
Let $M \in \Mod \Lambda$ and $X = \begin{tikzcd}[cramped, sep=small]X^{-1} \arrow[d, "f"] \\ X^0
\end{tikzcd} \in \Kproj$. Then $M \in \Ww(X)$ (or equivalently $X \in \Tt(M)$) if and only if the $\Bbbk$-linear map
\begin{equation*}\label{inducedmorf} \Hom_\Lambda(X^0, M) \xlongrightarrow{\Hom_\Lambda(f, M)} \Hom_\Lambda(X^{-1},M) \end{equation*}
is an isomorphism. By definition, $X \in \prescript{\perp}{}M[i]$ (or equivalently $M \in X[-i]^{\perp}$) for any $i \neq 0, 1$, so the only thing left to  prove is the case when $i = 0 $ or $1$. Consider the triangle \[ X[-1]  \dashrightarrow X^{-1} \xrightarrow{f} X^0 \rightarrow X . \] 
Then by applying $\Hom_{\D}(-, M)$ we get an exact sequence 
\begin{equation*}
	\begin{tikzcd} \Hom_{\D}(X, M) \arrow[r] & \Hom_{\D}(X^0, M) \arrow[r , "{\Hom_{\D}(f,M)}"] \arrow[d, equal] & [3em]\Hom_{\D}(X^{-1}, M) \arrow[r] \arrow[d, equal] & \Hom_{\D}(X[-1], M) \\	& \Hom_{\Lambda}(X^0, M) \arrow[r, "{\Hom_\Lambda(f,M)}"] & \Hom_{\Lambda}(X^{-1}, M) & \end{tikzcd}
\end{equation*}
This implies that $\Hom_{\Lambda}(f,M)$ is an isomorphism if and only if $ \Hom_{\D}(X[-1], M) \simeq 0 \simeq \Hom_{\D}(X, M)$, which is equivalent to $M \in X^{\perp_{\mathbb{Z}}}$ (and $X \in \prescript{\perp_{\mathbb{Z}}}{}M$).
\end{proof}

\begin{remark}
	The wide subcategories $\Ww(\HH)$ were already considered in work of L.~Angeleri~H\"{u}gel, F.~Marks and J.~Vit\'{o}ria \cite{angeleri2016silting, angeleri2016ring}. In their work, they are defined with respect to a morphism between two (not necessarily finite-dimensional) projective $\Lambda$-modules, and are key to their generalization of large tilting modules and support $\tau$-tilting modules. 
\end{remark}

\begin{lemma}\label{perpsimple}
	Let $\Gamma$ be a non-positive dg algebra over $\Bbbk$ such that $H^0(\Gamma)$ is finite-dimensional. Let $S$ be a simple module in the heart of the standard t-structure $\{X \in \D(\Gamma)\ | \ H^i(X) = 0 \ \text{for } i \neq 0 \} \subset \D(\Gamma)$. Then $X \in \per(\Gamma)$ belongs to $\prescript{\perp_\mathbb{Z}}{}{S}$ if and only if $X \in \thi_{\per \Gamma}(\add(\Gamma) \setminus \add(P))$ where $P$ is the direct summand of $\Gamma$ satisfying that $\Hom_{\D(\Gamma)}(P, S[i])$ is a division algebra if $i = 0$ and is $0$ otherwise. 
\end{lemma}

\begin{proof}
We are going to prove that if $ X \in \prescript{\perp_\mathbb{Z}}{}{S} \cap \per\Gamma$ then $X \in \thi_{\per \Gamma}(\add(\Gamma) \setminus \add(P))$. The other direction is straightforward. Let $X \in \per \Gamma$, then by \cite[Lemma 2.14]{plamondon2011cluster} $X$ is quasi-isomorphic to a \textit{twisted complex} in $\D(\Gamma)$, that is, its underlying graded module is of the form 
\[\bigoplus_{i = 1}^l Q_i\]
where $Q_i \in \add(\Gamma)[n_i]$ for some $n_i \in \mathbb{Z}$ and its differential is given by
\[\begin{pmatrix}
	d_1 & f_{12} & \cdots & f_{1l} \\
	0 & d_2 & \cdots & f_{2l} \\
	\vdots & \vdots & \ddots & \vdots \\
	0 & 0 & \cdots & d_{l}
\end{pmatrix}\]
where $d_i$ is the differential corresponding to $Q_i$. Recall that for any $Q \in \add(\Gamma)[n]$, then $\Hom(Q,S[i]) = 0$ if and only if $Q \notin \add(P)$ or $i \neq n$. Since $\Hom(X, S[i]) = 0$ for all $i \in \mathbb{Z}$, we conclude that none of the $Q_l$ are shifts of $P$. In particular, $X \in \thi_{\per \Gamma}(\add(\Gamma) \setminus \add(P))$.
\end{proof}

\begin{proposition}\label{WwTt}
	Let 
	\begin{gather*}
		\Ww : \thi \Lambda \rightarrow \wide \Lambda \\ \Tt : \wide \Lambda \rightarrow \thi \Lambda
	\end{gather*}
	be the maps defined in \cref{W_T_perp}. Then $\Ww$ and $\Tt$ induce mutually inverse bijections between the set of left finite wide subcategories in $\Mod \Lambda$ and thick subcategories with enough injectives in $\Kproj$. 
\end{proposition}
\begin{proof}
	Let $\HH= \thi_{[-1,0]}(U)$ where $U$ a basic injective generator of $\HH$. Then by \cref{W_T_perp} we know that $\Ww(\HH) = \Ww(U) = U^{\perp_\mathbb{Z}} \cap \Mod \Lambda$. By \cref{theo-thick-wide}, we know that $\Ww(\HH) = \W_{(M,P)}$ where $(M,P)$ is the support $\tau$-tilting pair associated to the presilting complex $U$. The only thing we need to prove is that $\Tt(\W_{(M,P)}) \subset \thi_{[-1,0]}(U)$, since the other inclusion is always satisfied.
	
	Recall that since $U$ is a $2$-term presilting complex, it can be completed into a $2$-term silting complex. Let $T_U$ the Bongartz completion of $U$. Since $U$ is injective in $\HH = \thi_{[-1,0]}(U)$, then $T_U = U \oplus V$ where $V$ is the Bongartz complement of $U$ satisfying that $\add(V) \cap \add(U)$ \cite[Lemma 3.11]{garcia2023thick}. Let $\Ss \subset \D^{[-1,0]}(\Mod \Lambda)$ be the simple-minded collection associated to $T_U$, then $\Ss = \Ss_U \sqcup \Ss_V$ where $\Ss_U = U^{\perp_\mathbb{Z}} \cap \Ss$ and $\Ss_V = V^{\perp_\mathbb{Z}} \cap \Ss$. By \cite[Theorem 2.3]{asai2020semibricks}, we know that $\Ss_U \subset \Mod \Lambda$ and that $\Ww(U) = \W_{(M,P)} = \Filt(\Ss_U)$. Thus $\Tt(\Ww(\HH)) = \prescript{\perp_\mathbb{Z}}{}{\Ss_U} \cap \Kproj$.
	Recall that there exists a non-positive dg algebra $\Gamma_U$ and triangulated equivalences
	\begin{center}
		\begin{tikzcd}
			\D_{fd}(\Gamma_U) \arrow[r, "\phi"', "\simeq"] & \D^b(\Mod \Lambda) \\
			\per(\Gamma_U) \arrow[u, hook] \arrow[r, "\bar{\phi}"', "\simeq"] & \Kb \arrow[u, hook]	
		\end{tikzcd}
	\end{center}
	which take $\Gamma_U$ to $T_U$ and all simple $H^0(\Gamma_U)$-modules to the simple-minded collection $\Ss$ (\cref{silt_smc}). Let $\bar{U} \in \add(\Gamma_U)$ be the perfect complex sent to $U$ and $\overline{\Ss_U}$ the set of simple $H^0(\Gamma_U)$-modules that are perpendicular to $\bar{U}$. Then $\phi(\overline{\Ss_U}) = \Ss_U$ and thus \[ \begin{split}
		\Tt(\Ww(\HH)) & = \prescript{\perp_\mathbb{Z}}{}{\Ss_U} \cap \Kproj \\ & =\prescript{\perp_\mathbb{Z}}{}{\phi\left(\overline{\Ss_U}\right)} \cap \Kproj \\ &= \phi\left(\prescript{\perp_\mathbb{Z}}{}{\overline{\Ss_U}} \right) \cap \Kproj
	\end{split}\ \] 
By \cref{perpsimple}, we know that $\prescript{\perp_\mathbb{Z}}{}{\overline{\Ss_U}} = \thi_{\per(\Gamma_U)}(\bar{U})$. Since $\phi$ is a triangle equivalence $\phi(\thi_{\per(\Gamma_U)}(\bar{U})) = \thi_{b}(U)$. This implies that $\Tt(\Ww(\HH)) = \thi_{b}(U) \cap \Kproj$. By applying \cref{explicitthick} we get that \[ \Tt(\Ww(\HH)) = \thi_{[-1,0]}(U) = \HH. \]
\end{proof}

\begin{theorem}\label{allthick}
	Suppose $\Lambda$ is $g$-finite. Then all thick subcategories of $\Kproj$ have enough injectives. 
\end{theorem}
\begin{proof}
	 Let $\HH$ be a thick subcategory of $\Kproj$, by \cref{gfinite-thick} we know that there exists a $2$-term presilting complex $U$ such that $\HH = \thi_{[-1,0]}(U)$. Since $\Lambda$ is $g$-finite, then $\Ww(\HH)$ is left finite and by \cref{WwTt} $\Tt(\Ww(\HH))$ has enough injectives. Let $V$ be an injective generator of $\Tt(\Ww(\HH))$, then
	 \[\thi_{[-1,0]}(U) \subset \Tt(\Ww(\thi_{[-1,0]}(U))) = \thi_{[-1,0]}(V).\] 
	 Since $U \in \thi_{[-1,0]}(V) = \CCone(V,V)$ \cite[Lemma 3.9]{garcia2023thick}, there exists a conflation $U \rightarrowtail V' \twoheadrightarrow V''$ with $V', V'' \in \add(V)$ and thus $U$ is $2$-term presilting in $\thi_b(V)$ with respect to $V[-1]$. By \cref{Bconp_silting}, we know that $U$ can be completed into a silting complex in $\thi_b(V)$. Since $U$ and $V$ give rise to the same $\tau$-perpendicular category $\Ww(\HH)$, we deduce that $|U| = |V| = |V[-1]|$, which in turn implies by \cref{maximal_silt} that $U$ is already silting in $\thi_b(V)$ and thus $\thi_b(V) = \thi_b(U)$. By \cref{thickisthick} this implies that $$\HH = \thi_{[-1,0]}(U) = \thi_{[-1,0]}(V).$$
\end{proof}

\begin{remark}
	The arguments used both in \cref{WwTt} and \cref{allthick} can be adapted to show that the results hold if we substitute left finite for \textit{right finite} wide subcategories \cite[Definition 1.2]{asai2020semibricks}. That is, $\Ww$ and $\Tt$ are inverse of each other if restricted to the set of right-finite wide subcategories and thick subcategories with enough projectives. Moreover, if $\Lambda$ is $g$-finite, the dual statement of \cref{allthick} also holds.
\end{remark}

\begin{theorem}\label{allthickproj}
	Suppose $\Lambda$ is $g$-finite. Then all thick subcategories of $\Kproj$ have enough projectives. 
\end{theorem}

\begin{corollary}\label{allthickprojinj}
		Suppose $\Lambda$ is $g$-finite. Then all thick subcategories of $\Kproj$ have enough projectives and enough injectives. 
\end{corollary}


\bibliographystyle{alpha}
\bibliography{Biblio}

\end{document}